\documentclass[10pt,twoside,a4paper]{article}
\usepackage{amsmath,amssymb,mathrsfs}
\usepackage{amsthm}

\usepackage{inputenc}
\usepackage{verbatim}
\usepackage{avant}
\usepackage[english]{babel}
\usepackage{thmtools,thm-restate}
\usepackage[nottoc]{tocbibind}
\usepackage{graphicx}
\usepackage[shortlabels,inline]{enumitem}
\usepackage{tikz}
\usepackage{bbm}
\usepackage{esint}
\usepackage{authblk}
\usepackage{mathtools}
\usepackage{microtype}

\usepackage{nicematrix}
\NiceMatrixOptions{cell-space-limits = 1pt}

\usepackage[inner=2.4cm,outer=2.4cm,top=2.8cm,bottom=2.8cm]{geometry}
\usepackage[all]{xy}
\usepackage{todonotes}

\usepackage[bookmarks=true]{hyperref}
\usepackage{xcolor}
\hypersetup{
    colorlinks,
    linkcolor={red!50!black},
    citecolor={blue!50!black},
    urlcolor={blue!80!black},
}

\mathtoolsset{showonlyrefs}

\DeclareMathOperator{\id}{id}
\DeclareMathOperator*{\argmin}{argmin}

\DeclareMathOperator{\hess}{Hess}
\newcommand{\scal}[2]{\ensuremath{\langle #1 , #2 \rangle}} 
\newcommand{\norm}[1]{\left\lVert#1\right\rVert}

\newcommand{\Leb}{\mathscr{L}}

\newcommand{\N}{\mathbb{N}}
\newcommand{\R}{\mathbb{R}}

\newcommand{\de}{\ensuremath{\, \mathrm d}} 

\newcommand\restr[2]{{
  \left.\kern-\nulldelimiterspace 
  #1 
  \right|_{#2} 
  }}

\newcommand{\RCD}{\mathsf{RCD}}

\newcommand{\X}{\mathsf{X}}

\newcommand{\di}{\mathsf d} 

\DeclareMathOperator{\evi}{EVI}
\DeclareMathOperator{\edi}{EDI}
\newcommand{\evikn}{\evi_{K,N}}
\newcommand{\evion}{\evi_{0,N}}
\newcommand{\EVIkn}{\mathbb{EVI}_{K,N}}
\newcommand{\EVI}{\mathbb{EVI}}

\newcommand{\rr}{\mathcal{R}}

\newcommand{\skn}{\mathfrak s_{K,N}}
\newcommand{\ckn}{\mathfrak c_{K,N}}

\newcommand{\eqdef}{\coloneqq}
\newcommand{\loc}{\mathrm{loc}}

\newcommand{\paren}[1]{\left(#1\right)}
\newcommand{\tparen}[1]{\big(#1\big)}
\newcommand{\braket}[1]{\left[#1\right]}
\newcommand{\tbraket}[1]{\big[#1\big]}

\newcommand{\abs}[1]{\left\lvert#1\right\rvert}
\newcommand{\seq}[1]{\left(#1\right)}

\usepackage{titlesec}
\titleformat{\section}
  {\scshape\large\centering}
  {\thesection}{0.5em}{}
\titleformat{\subsection}
  {\scshape\centering}
  {\thesubsection}{0.4em}{}

\title{\textbf{Gradient flows of $(K,N)$-convex functions\\with negative $N$}}
\date{\today}

\author{Lorenzo Dello Schiavo\footnote{Dipartimento di Matematica -- Universit\`a degli Studi di Roma ``Tor Vergata''. \textit{E-mail}: \href{mailto:delloschiavo@mat.uniroma2.it}{delloschiavo@mat.uniroma2.it}}, Mattia Magnabosco\footnote{Mathematical Institute, University of Oxford. \textit{E-mail}:  \href{mailto:mattia.magnabosco@maths.ox.ac.uk}{mattia.magnabosco@maths.ox.ac.uk}} \ and Chiara Rigoni\footnote{Faculty of Mathematics, University of Vienna. \textit{E-mail}: \href{mailto:chiara.rigoni@univie.ac.at}{chiara.rigoni@univie.ac.at}}}

\newtheoremstyle{remark}
        {10pt}
        {10pt}
        {}
        {}
        {\itshape}
        {.}
        {.4em}
        {}

\newtheoremstyle{proof}
        {10pt}
        {10pt}
        {}
        {}
        {\itshape}
        {.}
        {.4em}
        {}
        
\newtheoremstyle{definition}
        {10pt}
        {10pt}
        {}
        {}
        {\bfseries}
        {.}
        {.4em}
        {}

\newtheoremstyle{theorem}
        {10pt}
        {10pt}
        {\slshape}
        {}
        {\bfseries}
        {.}
        {.4em}
        {}

\theoremstyle{theorem}
\newtheorem{theorem}{Theorem}[section]
\newtheorem{prop}[theorem]{Proposition}
\newtheorem{corollary}[theorem]{Corollary}
\newtheorem{lemma}[theorem]{Lemma}

\theoremstyle{definition}
\newtheorem{assumption}[theorem]{Assumption}
\newtheorem{definition}[theorem]{Definition}

\theoremstyle{remark}
\newtheorem{example}[theorem]{Example}
\newtheorem{remark}[theorem]{Remark}

\theoremstyle{proof}
\newtheorem*{pro}{Proof}
 {\popQED\end{pro}}

\makeatletter
\renewcommand\xleftrightarrow[2][]{%
  \ext@arrow 9999{\longleftrightarrowfill@}{#1}{#2}}
\newcommand\longleftrightarrowfill@{%
  \arrowfill@\leftarrow\relbar\rightarrow}
\makeatother

\makeindex
\begin{document}

\maketitle
\begin{abstract}
We discuss $(K,N)$-convexity and gradient flows for $(K,N)$-convex functionals on metric spaces, in the case of real~$K$ and \emph{negative}~$N$. 
In this generality, it is necessary to consider functionals unbounded from below and/or above, possibly attaining as values both the positive and the negative infinity.
We prove several properties of gradient flows of $(K,N)$-convex functionals characterized by Evolution Variational Inequalities, including contractivity, regularity, and uniqueness.
\end{abstract}


\section{Introduction}

This paper aims to explore the properties of gradient flows for $(K, N)$-convex functions, specifically when the parameter $N$ is negative, within the framework of metric spaces.

\paragraph{Gradient flows} 
The theory of gradient flows was developed to generalize the classical notion of gradient-driven evolution beyond smooth and linear spaces, such as Hilbert spaces, to broader settings like Banach or metric spaces. One of the most widely studied and influential theories, with significant applications to various types of PDEs, involves the evolution of systems within a Hilbert space $X$, driven by a lower semicontinuous and convex (or $K$-convex) functional $\varphi \colon X \to (- \infty, +\infty]$
 with non-empty, proper domain  $\mathcal D[\varphi] \coloneqq \{ y \in X : \varphi (y) < \infty\}$. In this case, the evolution can be described by a locally Lipschitz curve $u \colon (0, +\infty) \to \mathcal{D}[\varphi]$ solving the differential inclusion
 \begin{equation}\label{eq:difincl}
 u'(t) \in -\partial \varphi\big(u(t)\big) \quad \text{ for } \mathcal{L}^1\text{-a.e. } t > 0 \quad \text{ with } \quad \lim_{t \downarrow 0} u(t) = u_0 \in \overline{\mathcal{D}[\varphi]},
 \end{equation}
where $\partial \varphi$ is the Fr\'echet subdifferential. The existence, uniqueness, well-posedness, and regularity of solutions to  \eqref{eq:difincl} have been examined in several foundational works (see e.g. \cite[Chs. III, IV]{BrezisBook}).

Among the many important properties of the solution $u(t)$ to \eqref{eq:difincl}, we recall that the energy map $t \mapsto \varphi \big(u(t)\big)$ is absolutely continuous and satisfies the following \emph{energy dissipation identity}
\begin{equation}\label{eq:EDIIntro}
\frac{\de}{\de t} \varphi\big(u(t)\big) = - \norm{ u'(t) }_X^2 = - \norm{ \partial \varphi\big(u(t)\big) }^2_X, \qquad \text{ for } \mathcal L^1\text{-a.e. } t > 0,
\end{equation}
$\norm{ (\partial \varphi)\circ u }_X$ being the minimal norm of the elements in $ \varphi\big(u(\cdot)\big)$.
Moreover, every solution $u(\cdot)$ arises from the locally-uniform limit of the discrete approximations obtained by the \emph{implicit Euler scheme} (also: \emph{proximal point method})
\begin{equation}\label{eq:Untau}
U^n_\tau = J_\tau \big( U^{n-1}_\tau \big), \quad n \in \N, \quad \text{ with } U^0_\tau = u_0,
\end{equation}
where the resolvent map $J_\tau \colon X \to X$ is defined by setting
\[
U = J_\tau (V) \quad \overset{\text{def}}{\iff} \quad \dfrac{U - V}{\tau} \in -\partial \varphi(U).
\]
For any sufficiently small time step $\tau > 0$, \eqref{eq:Untau} recursively defines a sequence $\{ U^n_\tau \}_{n \in \N}$, which in turns induces a piecewise constant interpolant
\begin{equation}\label{eq:Ubar}
\bar U_\tau (t) := U^n_\tau = J^n_\tau(u_0) \text{ if } t \in \big((n-1) \tau, n \tau\big],
\end{equation}
converging to the solution $u$ of \eqref{eq:difincl} when $\tau \downarrow 0$.

Following the pioneering work by De Giorgi, Degiovanni, Marino, and Tosques \cite{DeGMaTo, DeMaTo}, this theory has evolved in two significant ways: firstly by relaxing the convexity requirements on $\varphi$ (see e.g. \cite{RoSa, MiRoSa}), and secondly by expanding the framework from Hilbert spaces to Banach spaces, or, even more generally, to metric and topological spaces \cite{DeG}. We also refer to \cite{Sant, DaSa} for more recent overviews. In particular, we remark that the foundational work by De Giorgi and his collaborators effectively integrates both of these developments.

As for the second direction, we recall that there are several different formulations of gradient flows in a metric space $(X, \di)$. The first one we mention, introduced in \cite{DeGMaTo}, takes inspiration from \eqref{eq:EDIIntro} and serves as an equivalent characterization of \eqref{eq:difincl} in Hilbert spaces. The starting point consists in introducing the notion of \emph{metric derivative}
\[
|\dot u|(t) := \lim_{h \to 0} \dfrac{\di(u(t+h), u(t))}{|h|},
\]
which serves as a metric counterpart of the (scalar) velocity of a curve $u \colon [0, \infty) \to X $, and the one of \emph{descending slope}
\[
|D^- \varphi| (u) := \limsup_{z \to y} \dfrac{\big( f(z) - f(y)  \big)^-}{\di (z, y)},
\]
which plays the role of the norm of the (minimal selection in the) subdifferential. Hence we can introduce the notion of \emph{curve of maximal slope} as an absolutely continuous curve $u \colon [0, \infty) \to X$ satisfying the \emph{energy dissipation equality} (EDE)
\[
\frac{\de}{\de t } \varphi\big( u(t) \big) = |\dot u|^2(t) = - |D^- \varphi|^2\big(u(t)\big) \text{ for } \mathcal L^1\text{-a.e. } t > 0,
\]
which exactly provides the metric formulation of \eqref{eq:EDIIntro}.

The second approach is instead related to a variational formulation of \eqref{eq:Untau}, which can be stated in terms of the following variational condition
\begin{equation}\label{eq:VarUn}
U \in J_\tau(V) \quad \Longleftrightarrow \quad U \in \argmin_{W \in X} \dfrac{1}{2\tau} \di^2(V, W) + \varphi(W)
\end{equation}
whose minimizers define a multivalued map, still denoted by $J_\tau$.
A recursive selection of $U_\tau^n$ among the minimizers $J_\tau(U^{n-1})$ of \eqref{eq:VarUn} provides a useful algorithm to approximate gradient flows. In particular, pointwise limits (up to subsequences) as $\tau \downarrow 0$ of the interpolants $U_\tau$, defined through~\eqref{eq:Ubar},  are called \emph{Generalized Minimizing Movements}.

Finally, the most powerful notion of gradient flow is given by a metric version of \eqref{eq:difincl}, which can be formulated as a system of \emph{evolution variational inequalities} ($\evi$). In order to obtain this, we start observing that for any $w \in X$ it holds that
\[
\dfrac12 \frac{\de}{\de t} \norm{ u(t) - w }_X^2 = \langle u'(t),  u(t) - w \rangle_X
\]
and, in the case in which $\varphi$ is $K$-convex, using the properties of the subdifferential, we get that
\[
\dfrac12 \frac{\de}{\de t} \norm{ u(t) - w }_X^2 \le \varphi (w) - \varphi(u(t)) - \dfrac{K}{2} \norm{ u(t) - w }_X^2  \quad \text{ for every } w \in \mathcal D(\varphi).
\]
In metric spaces, one can seek curves $u \colon [0, \infty) \to \mathcal{D}(\varphi)$ that satisfy
\[
\dfrac12 \frac{\de}{\de t} \di^2(u(t), w) + \dfrac{K}{2} \di^2(u(t), w) \le \varphi (w) - \varphi(u(t))  \quad \text{ for every } w \in \mathcal D(\varphi).
\]
If solutions exist for all initial data $u_0 \in \mathcal D(\varphi)$, they generate a $K$-contracting semigroup $(S_t)_{t \ge 0}$ and this is known as the $\evi_K$ formulation of the gradient flow driven by $\varphi$, see \cite{AmbGigSav08}.\\

\paragraph{$(K,N)$-convexity and curvature-dimension conditions} The study of EVI gradient flows has played a central role in characterizing the synthetic curvature-dimension condition in the context of metric measure spaces.
In particular, it has been essential in linking the Eulerian and Lagrangian approaches within this theory.  A first step in this direction was taken in \cite{GKO13}, where the correct link between the ${\sf CD}(K, \infty)$ condition and the Bochner inequality was established in the setting of Alexandrov spaces. Building on the ideas introduced in that work, \cite{AGS14} subsequently showed that, in an $\mathsf{RCD}(K, \infty)$ space, the gradient flow of the relative entropy exists for suitable initial data and satisfies
the $\evi_K$ condition. In turn, this result allows one to deduce the Bakry--\'Emery formulation of the $\mathsf{RCD}(K, \infty)$ condition, and ultimately to derive the Bochner inequality in a weak sense.

The class of $(K, N)$-convex functions for $N > 0$, strengthening $K$-convexity, was first introduced in \cite{ErbarKuwadaSturm} in order to study the properties and the consequences of a new curvature-dimension condition, known as the \emph{entropic curvature-dimension condition} $\mathsf{CD}^e(K, N)$. In this setting, a function $f \colon I \to (-\infty, +\infty]$ on an interval $I \subset \R$ is $(K, N)$-convex if the inequality
\[
f'' \ge K + \dfrac{(f')^2}{N}
\]
holds in a distributional sense.
A function $F$ on a geodesic space $(X, \di)$ is called $(K, N)$-convex if, for every pair of points in $X$, there exists a unit-speed geodesic connecting them along which~$F$ is $(K, N)$-convex. 
With this definition, the $(K, N)$-convexity is a weak (nonsmooth) formulation of the inequality
\[
\text{Hess} F \ge K + \dfrac1N (\nabla F \otimes \nabla F).
\]
The analysis of $(K, N)$-convex functions and their gradient flows led to a new form of the Evolution Variation Inequality $\evikn$, also taking into account the effect of the dimension bound. The notion of $\evikn$ gradient flow yields new, sharp contraction estimates (more accurately: expansion bounds)  for the heat flow in the Wasserstein space.
In particular, this allows to conclude that the $\RCD^\ast(K, N)$ condition is equivalent to the existence of an $\evikn$ gradient flow of the entropy, in a suitable sense \cite[Thm.~3.17]{ErbarKuwadaSturm}, and reinforces the established result that characterizes $\RCD(K,\infty)$ spaces by the fact that the heat flow is an $\evi_K$ gradient flow of the Boltzmann-Shannon entropy functional \cite[Thm.~9.3(iii)]{AGSCalc}.
Moreover, we recall that  in \cite{AMS19}  the $\mathsf{CD}^\ast(K, N)$ condition is characterized via the validity of $\evi$ for suitable functionals related to McCann's class ${\rm DC}(N)$ (see \cite{McCann, Villani}).
Finally we refer to \cite{Sturm18} for the study of the $\evi_K$ gradient flows of the relative entropy functional in the Wasserstein space $(\mathcal{P}_2(X), W_2)$ over an $\RCD(K, \infty)$ space~$(X, \di, \mathfrak{m})$, and to \cite{MurSav} for a comprehensive analysis of the class of gradient flows within a metric space $(X, \di)$, which can be characterized by Evolution Variational Inequalities.\\

\paragraph{The case of negative dimension}
The study of $\evikn$-gradient flows for $(K,N)$-convex functions was extended to cases where the parameter 
$N$ takes negative values
 by Ohta in \cite{ohta}, motivated by \cite{OTI, OTII}. It is important to note that $(K,N)$-convexity for $N < 0$ is a weaker condition than $K$-convexity, thereby encompassing a broader class of functions. In \cite{ohta}, $(K,N)$-convex functions are defined, and their properties are explored, including the evolution variational inequality along gradient curves in a Riemannian context (see \cite[Lem. 2.3]{ohta}). Additionally, regularizing estimates are derived from the evolution variational inequality, and an expansion bound for gradient flows of \emph{Lipschitz} $(K,N)$-convex functions on Riemannian manifolds is established (see \cite[Theorem 3.8]{ohta}). However, the approach proposed in \cite{ohta} does not clearly establish whether this expansion bound can be obtained for general $(K,N)$-convex functions without the Lipschitz requirement, which in turn would ensure the uniqueness of 
$\evikn$ gradient curves also in the case $N < 0$.

\paragraph{Main results}
The main goal of this paper is to study the fundamental properties of $\evikn$ gradient curves  for negative values of the parameter $N$, associated with lower semicontinuous functionals defined on a metric space. In particular, in Section \ref{sec4} we prove the uniqueness (assuming existence) of 
$\evikn$ gradient curves starting from arbitrary points in the domain, while in Section \ref{sec5}, we carefully examine the relationship between the existence of
$\evikn$ gradient curves and the existence of curves of maximal slope that satisfy either the \emph{Energy Dissipation Equality} (EDE) or the \emph{Energy Dissipation Inequality} (EDI). The framework we adopt extends that of \cite{ohta} by defining $(K,N)$-convexity in a manner that includes functions taking the value $-\infty$.
This generalization is motivated by specific examples of $(K,N)$-convex functions defined on open intervals of 
$\R$, whose natural extensions to the boundary points take the value $-\infty$.
For instance,  for any $N < 0$ the function $f\colon x \mapsto -N\log x$ is $(0, N)$-convex in $(0, +\infty)$ and $\lim_{x \to 0} f(x) = -\infty$. In this case, the gradient flow of~$f$ satisfies~$\dot y_t= - f'(y_t) = -N/y_t$. Thus, for every starting point~$y_0\geq 0$, we have $y_t=\sqrt{y_0^2+2Nt}$ for $t\in \big[0,-\tfrac{y_0^2}{2N}\big)$ and~$y_t=0$ for~$t>-\tfrac{y_0^2}{2N}$.

This will be a crucial point when studying the $\evikn$-gradient flow of the Boltzmann entropy in ${\mathsf CD}^e(K, N)$-spaces, namely metric measure spaces satisfying the \emph{entropic curvature-dimension condition}.
Furthermore this example illustrates that we cannot generally expect a function~$\varphi$, $(K, N)$-convex for some~$N < 0$, to be coercive in the sense of \cite[Property 2.1b]{AmbGigSav08}, meaning that for each $\tau > 0$ and  $x \in X$ it holds
\[
\inf_{y \in X} \dfrac{1}{2\tau} \di^2(x, y) + \varphi(y) \equiv -\infty.
\]
The \emph{lack of coercivity}  is a delicate issue, as it prevents us from directly applying the standard theory of minimizing movement schemes to establish the existence of curves of maximal slope or EDI gradient flows for these functions.

Our analysis of the properties of gradient flows of $(K,N)$-convex functions surpasses the expansion bound established in \cite[Thm 3.8]{ohta}.
A key tool in this effort is the introduction of suitable \emph{reparametrization maps}.
These maps are central to our study, as they enable us to associate to each gradient curve for a function $f$ a gradient curve for the corresponding function $f_N\colon x \mapsto \exp(-f(x)/N)$ and vice versa, provided that $f$ is $(K, N)$-convex with $N < 0$.
This relationship is crucial, as demonstrated in Proposition \ref{prop:ftofN}, where we show that in the metric setting the $(K, N)$-convexity property of~$f$ can be lifted to a $\lambda$-convexity property of~$f_N$, for a suitable~$\lambda \in \R$. This generalizes a result proven in~\cite{ohta} in the smooth case.
Consequently, we can apply the well-studied properties of $\evi_K$
 gradient curves from~\cite{MurSav} to derive analogous results for~$\evikn$ gradient curves with $N < 0$.

\paragraph{Examples and outlook}
We conclude by recalling that the range of values for the dimension $N$ has been extended to include negative values also in the theory of $\mathsf{CD}(K, N)$ spaces, as presented in \cite{ohta}. 
This extension was motivated in the setting of weighted Riemannian manifolds by the work \cite{OTI,OTII}, where the convexity of a certain generalization of the relative entropy (inspired by information geometry, see~\cite[\S3.1]{OTI}) was characterized combining the condition $\text{Ric}_N \ge 0$ with the convexity of another weight function, with $N$ possibly taking negative values depending on the choice of entropy.
Indeed, the entropy functionals in~\cite{OTI} belong to the class of \emph{Bregman divergences}, which are dissimilarity functionals encoding dually flat Riemannian structures on spaces of probability densities, in the sense of information geometry (see e.g.~\cite[\S1.3.2, and p.~28]{Ama16}).

In fact, most of the results on this topic are obtained in the case of weighted Riemannian manifolds.
For example, in the Euclidean setting, a direct application of the results in \cite{BraLieb} ensures that given any convex measure~$\mu$ with full dimensional convex support and $C^2$ density $\Psi$, the space $(\R^n, \di_{\text{Eucl}}, \mu)$ satisfies the $\mathsf{CD}(0, N)$ condition for $1/N \in (-\infty, 1/n]$ (i.e., $N \in (-\infty, 0) \cup [n, \infty]$), in the sense that $\mathrm{Ric}_N \ge 0$. This class of measures, introduced by Borell in \cite{Borell}, extends the set of the so-called log-concave measures and it has been largely studied for example in \cite{Bobkov07, BobLed, Kol13}. In particular, following the terminology adopted by Bobkov, the case $N \in (-\infty, 0)$ corresponds to the ``heavy-tailed measures'' (see also \cite{BCR05}), identified by the condition that $1/ \Psi^{1/(n-N)}$ is convex. An explicit example of these measures is given by the family of Cauchy probability measures on $(\R^n, \di_{\text{Eucl}})$
\begin{equation}\label{eq:intro1}
\mu^{n, \alpha} = \dfrac{c_{n, \alpha}}{\big(1+\abs{x}^2\big)^{\frac{n + \alpha}{2}}} \, \de x, \quad \alpha > 0,
\end{equation}
where $c_{n, \alpha} > 0$ is a normalization constant.
It then follows that $(\R^n, \di_{\text{Eucl}}, \mu^{n, \alpha})$ is a $\mathsf{CD}(0, -\alpha)$ space.
Moreover, in \cite{Milman-NegSph} it is proved that the $n$-dimensional unit sphere~$\mathbb{S}^n$ equipped with the harmonic measure, namely the {hitting distribution}  by the {Brownian motion} started at $x \in \mathbb{S}^n$, (which can be equivalently described as the probability measure whose density is proportional to $\mathbb{S}^n \ni y \to 1 / | y - x |^{n+1}$) is a $\mathsf{CD}\big(n-1-(n+1)/4, -1 \big)$ space. More generally, Milman provides an analogue to the family \eqref{eq:intro1} of Cauchy measures in $\R^{n+1}$, showing that the family of probability measures on~$\mathbb{S}^n$ having density proportional to
\[
\mathbb{S}^n \ni y \mapsto \dfrac{1}{\abs{y-x}^{n + \alpha}} \quad \text{ for all } \abs{x} < 1,\ \alpha \ge -n, \text{ and } n \ge 2, 
\]
satisfies the curvature-dimension condition $\mathsf{CD}(n-1-\frac{n+\alpha}{4}, -\alpha)$. 

Following \cite{ohta}, the study of non-smooth $\mathsf{CD}(K, N)$ spaces has continued in \cite{MRS23,MR23}, where several important properties have been established. In particular, these works identify the class of metric measure spaces equipped with a \emph{quasi-Radon} measure as the natural framework to define the $\mathsf{CD}(K,N)$ condition.
Drawing an analogy with the positive-dimensional case $N>0$, our study of $\evikn$ gradient flow may provide a fundamental tool for bridging the Eulerian and Lagrangian approaches on lower curvature bounds. This, in turn, could yield a consistent definition of $\mathsf{RCD}(K,N)$ spaces for negative values of $N$, allowing the development of an analytic approach to the study of metric measure spaces equipped with singular (quasi-Radon) measures.

\section*{Acknowledgments}
Part of this work was conducted at \emph{Erwin Schr\"odinger International Institute for Mathematics and Physics (ESI)}, where all authors participated in the workshop ``Synthetic Curvature Bounds for Non-Smooth Spaces: Beyond Finite Dimensions''.
We express our sincere gratitude to ESI for providing an excellent working environment and for supporting this collaborative research. The authors express their gratitude to Karl-Theodor Sturm for many stimulating discussions.
The first author acknowledges support from the Italian Ministry of Education through the PRIN Department of Excellence MatMod@TOV (CUP: E83C23000330006).
The second author acknowledges support from the Royal Society through the Newton International Fellowship (award number: NIF$\backslash$R1$\backslash$231659). This research
was funded in part by the Austrian Science Fund (FWF),
Grants: \href{https://www.fwf.ac.at/en/research-radar/10.55776/ESP208}{DOI: 10.55776/ESP208}, 
\href{https://www.fwf.ac.at/en/research-radar/10.55776/ESP224}{DOI:10.55776/ESP224}, and
\href{https://www.fwf.ac.at/en/research-radar/10.55776/EFP6}{DOI: 10.55776/EFP6}]. For open access purposes, the authors have applied
a CC BY public copyright license to any author accepted manuscript
version arising from this submission.

\section{Preliminaries}\label{sec:2}
Everywhere in the following~$X$ is a non-empty set,~$\di$ is a distance on~$X$, and $\X\eqdef (X,\di)$ the resulting metric space, not necessarily complete.
A curve $\gamma\in C([0, 1], \X)$ is called \textit{geodesic} if 
\begin{equation}
    \di(\gamma(s), \gamma(t)) = |t-s| \cdot \di(\gamma(0), \gamma(1)) \quad \text{for every }s,t\in[0,1]\,.
\end{equation}
A metric space~$(X,\di)$ is called:
\begin{itemize}
\item \emph{strictly intrinsic} if for every pair of points~$x_0$,~$x_1$ in~$X$ there exists a geodesic curve~$\gamma\colon [0,1]\to \X$ with~$\gamma_0=x_0$ and~$\gamma_1=x_1$, cf.~\cite[Dfn.~2.1.20]{BurBurIva07};
\item \emph{geodesic} if it is complete and strictly intrinsic.
\end{itemize}
The necessity of considering non-complete spaces arise from the potential application to functionals on metric measure spaces satisfying the curvature-dimension condition~$\mathsf{CD}(K,N)$ for~$N<0$.

We shall work under the following standing assumption:

\begin{assumption}\label{a:Main}
$(X,\di)$ is a metric space satisfying
\begin{enumerate}[$(i)$]
\item\label{i:a:Main:1} $(X,\di)$ is strictly intrinsic;
\item\label{i:a:Main:2} the metric completion $(\overline X,\overline\di)$ of~$(X,\di)$ is a geodesic metric space.
\end{enumerate}
\end{assumption}

We note, if~$(X,\di)$ is additionally locally compact,~\ref{i:a:Main:2} in Assumption~\ref{a:Main} follows from~\ref{i:a:Main:1} by the Hopf--Rinow--Cohn-Vossen Theorem~\cite[Thm.~2.5.28]{BurBurIva07}.
We stress that if~$(X,\di)$ is not locally compact, this is not the case, cf.~\cite[Ex.~2.5.24]{BurBurIva07}.

\medskip

In the following, we will consider gradient curves of  a lower semicontinuous functional $f\colon \X \to \bar \R = [-\infty,+\infty]$. For such functionals we will denote by $\mathcal{D}[f] \subset X$ the (\emph{finiteness}) \emph{domain} of $f$, i.e. $\mathcal{D}[f] \eqdef  \{x\in X : f(x)\neq \pm\infty\}$, and by $\mathcal{D}^\star[f] \subset X$ the \emph{extended domain}, ${\mathcal{D}}^\star[f] \eqdef  \{x\in X : f(x)< +\infty\}$. We assume that $\mathcal D [f] \neq \emptyset$.

\begin{definition}[$\lambda$-convexity]\label{def:lambdaconv}
Fix~$\lambda \in \R$.
A function $f\colon X\to \bar R$ is said to be \emph{$\lambda$-convex} if for every $x_0,x_1 \in \mathcal{D}[f]$, there exists a geodesic $\gamma\colon [0,1]\to X$ connecting $x_0$ and $x_1$, such that
\begin{equation}\label{eq:lambdaconvinequality}
    f(\gamma_t) \leq (1-t) f(x_0) + t f(x_1) - \frac{\lambda}{2} t (1-t)\, \di(x_0,x_1)^2 \qquad  t\in [0,1].
\end{equation}
\end{definition}

On a smooth, connected and complete Riemannian manifold $(M,g)$, a $C^2$ function $f\colon M\to \R$ is $\lambda$-convex if and only if
\begin{equation}
    \hess f (v,v) \geq \lambda|v|^2 \qquad \text{for all }v\in TM.
\end{equation}
This condition can be reinforced by introducing a dimensional parameter. In particular, given any $K\in \R$ and $N>0$, a $C^2$ function $f\colon M\to \R$ is said to be \emph{$(K,N)$-convex} if
\begin{equation}\label{eq:KNdiff}
    \hess f (v,v) - \frac{\scal{\nabla f}{v}^2}{N}\geq K|v|^2 \qquad \text{for all }v\in TM.
\end{equation}
Observe that the classical $\lambda$-convexity can be seen as a degenerate case of the $(K,N)$-convexity when $K=\lambda$ and the parameter $N$ takes the value $+\infty$. Moreover, as it happens in the classical case with Definition \ref{def:lambdaconv}, the differential condition \eqref{eq:KNdiff} has a metric counterpart, which turns out to be equivalent in the smooth setting. We refer the reader to \cite{ErbarKuwadaSturm} for a comprehensive introduction to the theory of $(K,N)$-convex functions in metric spaces for~$N>0$.

\subsection{$(K,N)$-convex functions for $N < 0$}

As the differential inequality \eqref{eq:KNdiff} makes sense also when $N$ is negative, it is meaningful to extend the definition of $(K,N)$-convex function to the case $N<0$.
Given $K\in \R$ and $N<0$, a $C^2$~function $f\colon M\to \R$ on a smooth, connected and complete Riemannian manifold $(M,g)$ is $(K,N)$-convex if \eqref{eq:KNdiff} holds.
In particular, $(K,N)$-convexity with negative $N$ is weaker than any $(K,N)$-convexity for positive $N$ and it is also weaker than $K$-convexity (in the sense of Definition~\ref{def:lambdaconv}), cf.~Proposition~\ref{prop:properties} below.
The differential definition~\eqref{eq:KNdiff} can be equivalently rephrased in terms of the function
\begin{equation}\label{eq:fN}
    f_N(x)\coloneqq e^{-f(x)/N}.
\end{equation}
Indeed, $f$ is $(K,N)$-convex if and only if 
\begin{equation}\label{eq:KNdiff2}
    \hess f_N (v,v) \geq -\frac K N f_N(x)|v|^2 \qquad \text{for all }v\in T_x M,\, x\in M. 
\end{equation}

As in the case of positive $N$, it is possible to give a metric definition of $(K,N)$-convexity, which is equivalent to \eqref{eq:KNdiff2} in the smooth setting, and thus properly extends the notion of $(K,N)$-convexity to nonsmooth spaces also in the case~$N<0$. The definition for any $K\in \R$ and $N<0$ is given in terms of the following functions
\begin{equation*}
    \skn,\ckn\colon [0,+\infty) \to \R, 
\end{equation*}
\begin{equation*}
    \skn (\theta)\eqdef  \begin{cases}
        \frac{1}{\sqrt{K/N}} \sin(\theta\sqrt{K/N}) &\text{if }K<0,\\
        \theta &\text{if }K=0,\\
         \frac{1}{\sqrt{-K/N}} \sinh(\theta\sqrt{-K/N}) &\text{if }K>0,\\
    \end{cases} \qquad \ckn (\theta)\eqdef  \begin{cases}
        \cos(\theta\sqrt{K/N}) &\text{if }K<0,\\
        1 &\text{if }K=0,\\
          \cosh(\theta\sqrt{-K/N}) &\text{if }K>0,\\
    \end{cases}
\end{equation*}
which are used to define the distortion coefficients
\begin{equation}\label{E:sigma}
\sigma_{K,N}^{(t)}(\theta)\eqdef  
\begin{cases}
\displaystyle  \frac{\skn (t\theta)}{\skn (\theta)} & \textrm{if}\  N\pi^{2} < K\theta^{2} <  0 \text{ or }K\theta^{2} > 0, \\
t & \textrm{if}\  
K \theta^{2}=0,  \crcr
\infty, & \textrm{if}\ K\theta^{2} \leq N\pi^{2}, \crcr
\end{cases}
\qquad \theta \in [0,\infty), \qquad t\in[0,1] ,
\end{equation}
and compute their derivatives.

These distortion coefficients have the following monotonicity properties (see \cite{Sturm06II}): for every fixed~$t\in [0,1]$ it holds
\begin{equation}\label{eq:monotonicity}
\begin{aligned}
    K&\longmapsto \sigma_{K,N}^{(t)}(\theta) \text{ is non-increasing}
    \\
    N&\longmapsto \sigma_{K,N}^{(t)}(\theta) \text{ is non-decreasing},
\end{aligned}
\qquad {\text{for each } \theta \in \begin{cases} 
[0,\pi\sqrt{N/K}) & \text{ if $K<0$,}
\\
[0,\infty) &\text{ if $K\geq 0$}.
\end{cases}
}
\end{equation}

\begin{definition}[$(K,N)$-convexity]\label{def:KNconv}
Fix~$K \in \R$ and $N <0$.
A function $f\colon X\to  \bar{\R}$ is said to be \emph{$(K,N)$-convex} if for every $x_0,x_1 \in {\mathcal{D}}^\star[f]$, with $\di\eqdef \di(x_0,x_1)<\pi\sqrt{N/K}$ when $K<0$, there exists a geodesic $\gamma\colon [0,1]\to X$ connecting $x_0$ and $x_1$, such that
\begin{equation}\label{eq:convinequality}
    f_N(\gamma_t) \leq \sigma_{K,N}^{(1-t)}(\di) f_N(x_0) + \sigma_{K,N}^{(t)}(\di) f_N(x_1), \qquad t\in [0,1],
\end{equation}
where we recall that $f_N(x)=e^{-f(x)/N}$, and, conventionally, we set~$0\cdot \infty = 0$.
\end{definition}

\begin{remark}
    Compared to the first (metric) definition of the $(K,N)$-convexity for $N < 0$, given by Ohta in \cite{ohta}, in Definition \ref{def:KNconv} we allow $(K,N)$-convex functions to attain the value $-\infty$. This has already been done in \cite{MRS23} and it was motivated by some examples of $(K,N)$-convex functions on bounded segments in $\R$, converging to $-\infty$ at the endpoints, cf.\ Example 2.4 in~\cite{ohta}.
\end{remark}

\begin{remark}
Let~$f\colon X\to\bar\R$ be $(K,N)$-convex for some~$K\in\R$ and~$N<0$, and denote by $\overline{f}\colon \overline{X}\to\bar\R$ the extension of~$f$ taking value~$+\infty$ on~$\overline{X}\setminus X$.
We may extend~\eqref{eq:convinequality} to~$\overline{f}$ and to every~$x_0,x_1\in\overline{X}$.
Indeed, since~$\overline{\X}$ is a geodesic space by Assumption~\ref{a:Main}, for every~$x_0,x_1\in\overline{X}$ there exists a geodesic~$\gamma\colon [0,1]\to\overline{X}$ connecting them.
If~$x_i\in\overline{X}\setminus \mathcal{D}^\star[\overline{f}]= \overline{X}\setminus\mathcal{D}^\star[f]$ for either~$i=0$ or~$i=1$, then~\eqref{eq:convinequality} holds trivially with~$\overline{f}$ in place of~$f$ since~$\overline{f}_N(x_i)=+\infty$.
If otherwise both~$x_0,x_1\in \mathcal{D}^\star[\overline{f}]=\mathcal{D}^\star[f]$, then~\eqref{eq:convinequality} for~$\overline{f}$ holds by~\eqref{eq:convinequality} for~$f$ since~$\overline{f}=f$ on~$\mathcal{D}^\star[\overline{f}]$.
In particular, this shows that if~$f\colon X\to\bar\R$ is $(K,N)$-convex, then so is its extension~$\overline{f}\colon \overline{X}\to\bar\R$.
\end{remark}

The following proposition lists some properties of $(K,N)$-convex functions, which were proved by Otha in \cite[Lem.s~2.6, 2.9]{ohta}.

\begin{prop}\label{prop:properties}
    Let $f\colon X\to  \bar{\R}$.
    If~$f$ is $K$-convex for some $K \in \R$, then it is $(K,N)$-convex for every $N<0$.
    If~$f$ is $(K,N)$-convex for some $K \in \R$ and $N <0$, then
    \begin{enumerate}[$(i)$]
        \item for every $a\in \R$, the function $f+a$ is $(K,N)$-convex;
        \item for every $c>0$, the function $cf$ is $(cK,cN)$-convex;
        \item $f$ is $(K',N')$-convex for every $K'\leq K$ and $N'\in [N,0)$;
    \end{enumerate}
\end{prop}

\begin{example}[{\cite[Ex.~2.4]{ohta}}]\label{ex:KNconvex}
Some examples of~$(K,N)$-convex functions with~$N<0$, and their domains, are as follows:
\begin{enumerate}[$(a)$]
\item\label{i:ex:KNconvex:1} for~$K>0$, the function $x\mapsto -N\log \tparen{\cosh(x\sqrt{-K/N})}$ on~$\R$;
\item\label{i:ex:KNconvex:2} for~$K>0$, the function $x\mapsto -N\log \tparen{\sinh(x\sqrt{-K/N})}$ on~$(0, +\infty)$;
\item\label{i:ex:KNconvex:3} for~$K=0$, the function~$x\mapsto -N\log x$ on~$(0, +\infty)$;
\item\label{i:ex:KNconvex:4} for $K<0$, the function~$x\mapsto -N\log\tparen{\cos(x\sqrt{K/N})}$ on~$\tparen{-\frac{\pi}{2}\sqrt{N/K},\frac{\pi}{2}\sqrt{N/K}}$.
\end{enumerate}
\end{example}

\begin{remark}[Lack of coercivity]\label{r:NoCoercive}
For~$N<0$, a~$(K,N)$-convex function need \emph{not} be coercive in any reasonable sense, irrespectively of the sign of~$K$.
For example, the functions in Example~\ref{ex:KNconvex}\ref{i:ex:KNconvex:2}--\ref{i:ex:KNconvex:4} are not coercive in the sense of~\cite[Eqn.~(2.1b), p.~43]{AmbGigSav08}, and their (proper) sublevel sets~$\{f\leq k\}\cap \mathcal D [f]$ are not even complete for any~$k\in\R$.

This fact has important implications concerning the possible (non-)existence of curves of maximal slope or \eqref{eq:edi} gradient flows for~$(K,N)$-convex functions, see Dfn.~\ref{d:EDI} below.
Indeed, since coercivity fails, one cannot apply the standard theory of minimizing movement schemes.
\end{remark}

In the smooth setting, $(K,N)$-convexity is a local property, since it is characterized by a pointwise differential inequality. The following proposition shows that the same holds in the nonsmooth setting.

\begin{prop}\label{prop:gluing}
    Fix $K\in \R$, $N<0$ and let $a,b,c,d \in \R$ such that $a<b<c<d$. Assume that a function $f\colon [a,d]\to \bar \R$ is $(K,N)$-convex when restricted to $[a,c]$ and to $[b,d]$, then it is $(K,N)$-convex on the whole interval $[a,d]$.
\end{prop}

\begin{proof}
    By assumption, we have that 
    \begin{equation}
         f_N(b)\leq \sigma_{K,N}^{(\frac{c-b}{c-a})}(c-a) f_N(a) + \sigma_{K,N}^{(\frac{b-a}{c-a})}(c-a) f_N(c),
    \end{equation}
    \begin{equation}
         f_N(c)\leq \sigma_{K,N}^{(\frac{d-c}{d-b})}(d-b) f_N(b) + \sigma_{K,N}^{(\frac{c-b}{d-b})}(d-b) f_N(d).
    \end{equation}
    Combining these two inequalities, we deduce that 
    \begin{equation}\label{eq:sigmaaprofusione}
        f_N(c)\leq \frac{\sigma_{K,N}^{(\frac{c-b}{c-a})}(c-a) \, \sigma_{K,N}^{(\frac{d-c}{d-b})}(d-b)}{1-\sigma_{K,N}^{(\frac{b-a}{c-a})}(c-a) \, \sigma_{K,N}^{(\frac{d-c}{d-b})}(d-b)} f_N(a) + \frac{ \sigma_{K,N}^{(\frac{c-b}{d-b})}(d-b)}{1-\sigma_{K,N}^{(\frac{b-a}{c-a})}(c-a) \, \sigma_{K,N}^{(\frac{d-c}{d-b})}(d-b)} f_N(b).
    \end{equation}
    On the other hand, for $K\neq 0$ we can use the product-to-sum identities for trigonometric and hyperbolic functions to obtain
    \begin{equation}
        \begin{split}
            \skn(c-a)\,\skn(d-b) - &\skn(b-a)\,\skn(d-c)
            \\
            &=\frac{\text{sgn}(K)}{2\sqrt{|K/N|}} \braket{ \ckn(b+d-a-c) - \ckn(c+d-a-b) }
            \\
            &= \skn(d-a)\,\skn(c-b).
        \end{split}
    \end{equation}
    The same identity holds by explicit computation also for $K=0$, indeed 
    \begin{equation}
        (c-a)(d-b) - (b-a)(d-c) =(d-a)(c-b).
    \end{equation}
    As a consequence, we deduce that 
    \begin{equation}
        \begin{split}
            1-\sigma_{K,N}^{(\frac{b-a}{c-a})}(c-a) \, \sigma_{K,N}^{(\frac{d-c}{d-b})}(d-b) &= \frac{\skn(c-a)\,\skn(d-b) - \skn(b-a)\,\skn(d-c)}{\skn(c-a)\,\skn(d-b)}\\
            &=\frac{\skn(d-a)\,\skn(c-b)}{\skn(c-a)\,\skn(d-b)}.
        \end{split}
    \end{equation}
    Since, by definition,
    \begin{equation}
        \sigma_{K,N}^{(\frac{c-b}{c-a})}(c-a) \, \sigma_{K,N}^{(\frac{d-c}{d-b})}(d-b)= \frac{\skn(c-b)\,\skn(d-c)}{\skn(c-a)\,\skn(d-b)} \quad \text{ and } \quad\sigma_{K,N}^{(\frac{c-b}{d-b})}(d-b)=\frac{\skn(c-b)}{\skn(d-b)},
    \end{equation}
    we deduce from \eqref{eq:sigmaaprofusione} that 
    \begin{equation}
        f_N(c)\leq \sigma_{K,N}^{(\frac{d-c}{d-a})}(d-a) f_N(a) +\sigma_{K,N}^{(\frac{c-a}{d-a})}(d-a) f_N(d).
    \end{equation}
    Finally, the thesis follows from Lemma 3.10 in \cite{MRS23}.
\end{proof}

Guided by the results of S.~Ohta in the smooth case, in the next proposition we investigate the convexity of~$f_N$ when~$f$
 is $(K,N)$-convex.
\begin{prop}\label{prop:ftofN}
    Let $(X, \di)$ be a metric space and $f\colon X \to \bar \R$ be a function on it.
    \begin{enumerate}[$(i)$]
        \item\label{i:p:ftofN:1} If $f$ is $(K,N)$-convex for some $K\geq 0$ and $N<0$, then $f_N=e^{-f/N}$ is convex.
        \item\label{i:p:ftofN:2} If $f$ is bounded above by $M$ and $(K,N)$-convex for some $K< 0$ and $N<0$, then $f_N=e^{-f/N}$ is $-\frac{K}{N} e^{-M/N}$-convex.
    \end{enumerate}
\end{prop}

\begin{proof}
    \ref{i:p:ftofN:1} Given any $x_0,x_1 \in X$, let $\gamma\colon [0,1]\to X$ be the constant speed geodesic connecting $x_0$ and $x_1$ and satisfying \eqref{eq:convinequality}, provided by Definition \ref{def:KNconv}. Then, for every $t\in [0,1]$ we have that 
    \begin{equation}
    f_N(\gamma_t) \leq \sigma_{K,N}^{(1-t)}(\di) f_N(x_0) + \sigma_{K,N}^{(t)}(\di) f_N(x_1)\leq (1-t) f_N(x_0) + t f_N(x_1),
\end{equation}
where the second inequality follows from \eqref{eq:monotonicity}, recalling that $\sigma_{0,N}^{(t)}(\di)=t$ and $\sigma_{0,N}^{(1-t)}(\di)=1-t$.

\medskip

\noindent \ref{i:p:ftofN:2} First of all, we observe that for every $t\in[0,1]$ it holds
\begin{equation}
    \sigma_{K,N}^{(t)}(\theta) = \frac{t \theta - \frac{1}{6} \frac{K}{N} t^3 \theta^3+ o (\theta^4) }{ \theta - \frac{1}{6} \frac{K}{N}  \theta^3+ o (\theta^4)}
\end{equation}
and therefore
\begin{equation}
    \sigma_{K,N}^{(t)}(\theta) -t = \frac{\frac{1}{6} \frac{K}{N} (t-t^3) \theta^3+ o (\theta^4) }{ \theta - \frac{1}{6} \frac{K}{N}  \theta^3+ o (\theta^4)}=  \frac{1}{6} \frac{K}{N} (t-t^3)  \theta^2 + o (\theta^3).
\end{equation}
Consequently, given any $\varepsilon>0$, there exists $\bar \theta$ such that for every $\theta\leq \bar \theta$ it holds 
\begin{equation}
    \sigma_{K,N}^{(t)}(\theta) -t\leq \frac{1}{6}\bigg[ \frac{K}{N} + \varepsilon\bigg] (t-t^3)  \theta^2 \quad \text{and} \quad \sigma_{K,N}^{(1-t)}(\theta) -t\leq \frac{1}{6}\bigg[ \frac{K}{N} + \varepsilon\bigg] \big(1-t-(1-t)^3\big)  \theta^2.
\end{equation}
At this point, consider any $x_0,x_1 \in X$ with $\di=\di(x,y)\leq \bar \theta$ and let $\gamma\colon [0,1]\to X$ be the constant speed geodesic connecting $x_0$ and $x_1$ and satisfying \eqref{eq:convinequality}, provided by Definition \ref{def:KNconv}. Then a direct computation gives
\begin{equation}
    \begin{split}
    f_N(\gamma_t) &\leq \sigma_{K,N}^{(1-t)}(\di) f_N(x_0) + \sigma_{K,N}^{(t)}(\di) f_N(x_1)\\
    &= (1-t) f_N(x_0) + t f_N(x_1)+ (\sigma_{K,N}^{(1-t)}(\di) - (1-t))f_N(x_0) + (\sigma_{K,N}^{(t)}(\di) -t) f_N(x_1),
    \end{split}
\end{equation}
and, taking into account the bound on $f$, we obtain
\begin{equation}
      f_N(\gamma_t) \leq (1-t) f_N(x_0) + t f_N(x_1)+ \big[(\sigma_{K,N}^{(1-t)}(\di) - (1-t)) + (\sigma_{K,N}^{(t)}(\di) -t) \big] e^{-M/N}.
      \end{equation}
Now, using the bounds on  $\sigma_{K,N}^{(t)}(\di) -t$ provided above, we get
\begin{equation}
    f_N(\gamma_t) \leq (1-t) f_N(x_0) + t f_N(x_1) + \frac{1}{2}\bigg[ \frac{K}{N} + \varepsilon\bigg] e^{-M/N} t(1-t) \di^2.
\end{equation}
This shows that $f_N$ is $-\big[ \frac{K}{N} + \varepsilon\big] e^{-M/N}$-convex along every sufficiently short geodesic, but according to Proposition \ref{prop:gluing} this is sufficient to conclude that $f_N$ is $-\big[ \frac{K}{N} + \varepsilon\big] e^{-M/N}$-convex. As $\varepsilon$ can be taken arbitrarily small, we conclude that $f_N$ is $-\frac{K}{N} e^{-M/N}$-convex.
\end{proof}

\begin{remark}
    We note that Proposition~\ref{prop:ftofN} gives the best uniform bounds we can expect, according to the smooth case, cf.~\eqref{eq:KNdiff}.
    For instance, this is the case in~\ref{i:p:ftofN:2} for the function $f\colon x\mapsto -N\log\tparen{\cos(x\sqrt{K/N})}$ in Example~\ref{ex:KNconvex}\ref{i:ex:KNconvex:4} with~$M\eqdef \max f =0$.
\end{remark}

\subsection{Gradient flows}\label{subsec2.2}

In this section we introduce the notion of $\evikn$ gradient curves and gradient flow for a lower semicontinuous functional on a metric space $(X,\di)$. As a starting point, we recall the definition of $\evi_\lambda$ gradient curve (and gradient flow), summarizing also the most relevant related facts.

In the following we will usually denote by $(y_t)_{t\in (0,T)}$ the curves in $X$ defined on $(0,T)$ and, given a function $f \colon I \to \R$ defined on some interval $I$, we denote by $\frac{\de^+}{\de t}f(t)$ the right derivative of $f$ at the point $t \in I$, namely
\[\frac{\de^+}{\de t}f(t) \eqdef \limsup_{h \downarrow 0} \dfrac{f(t+h) - f(t)}{h}.
\]
Moreover, we say that a curve $(y_t)_{t\in (0,T)}$ starts at $\bar y\in X$ if $\lim_{t\downarrow 0} y_t=\bar y$.

\begin{definition}[$\evi_\lambda$ gradient curve]
    Let $g\colon X \to \bar \R$ be a lower semicontinuous functional.
    Given $\lambda \in \R$, an \emph{$\evi_\lambda$ gradient curve for $g$} is a continuous curve $(z_t)_{t\in (0,T)} \subset \mathcal D[g]$, with $T\in (0,+\infty]$, such that 
    \begin{equation}\label{eq:EVIlambda}
        \frac{\de^+}{\de t} \frac{\di(z_t,z)^2}{2}   + \frac \lambda 2 \di(z_t,z)^2 \leq g(z)-g(z_t)
    \end{equation}
    for every $z\in \mathcal{D}[g]$ and every $t\in(0,T)$.
\end{definition}

The following result highlights an important property of $\evi_\lambda$ gradient curves, namely the so-called \emph{$\lambda$-contractivity}. As an immediate consequence, this implies the uniqueness of $\evi_\lambda$ gradient curves for a given initial condition.

\begin{prop}[$\lambda$-contractivity  and uniqueness of $\evi_\lambda$]\label{prop:EVI}
Let $g\colon X \to \bar \R$ be a proper lower semicontinuous functional and $\lambda \in \R$. Let $(z^1_t)_{t \in (0, T)}, (z^2_t)_{t \in (0, T)} \subset \mathcal{D}[g]$ be two $\evi_\lambda$ gradient curves for $g$. Then for every $0 < r < s < T$ it holds
\begin{equation}
\di \big(z^1_s, z^2_s\big)\le e^{-\lambda (s - r)} \di \big(z^1_r, z^2_r\big).
\end{equation}
In particular, for every $z_0 \in \overline{\mathcal{D}[g]}$ there exists at most one $\evi_\lambda$ gradient curve~$(z_t)_{t\in(0,T)}$ for $g$ satisfying the initial condition $\lim_{t\downarrow 0} z_t  = z_0$.
\end{prop}

\begin{proof}
See~\cite[Thm.~3.5]{MurSav}. 
\end{proof}

The next result shows that when $g\colon X\to\bar\R$ is $\lambda$-convex the inequality \eqref{eq:EVIlambda} defining $\evi_\lambda$ is a local condition in space, uniform in time. This result is likely, as $\evi_\lambda$ gradient curves play a central role in the characterization of gradient flows in general metric spaces. 

\begin{prop}\label{prop:globalization}
    Let $g\colon X \to \bar \R$ be a lower semicontinuous, $\lambda$-convex functional and let $(z_t)_{t\in (0,T)} \subset \mathcal{D}[g]$, with $T\in (0,+\infty]$, be a continuous curve in $(X,\di)$.
    Assume that for every $t\in (0,T)$ there exists a neighborhood $Z_t$ of $z_t$ such that \eqref{eq:EVIlambda} holds for every $z\in Z_t \cap \mathcal{D}[g]$.
    Then, $(z_t)_{t\in (0,T)}$ is an $\evi_\lambda$ gradient curve for $g$.
\end{prop}

\begin{proof}
    We fix any $z\in \mathcal{D}[g]$ and $t\in (0,T)$. We consider the geodesic $(x_t)_{t\in [0,1]}$ connecting $z_t=x_0$ and $z=x_1$, along which the convexity inequality \eqref{eq:lambdaconvinequality} holds. Hence let us take $r\in(0,1]$ such that the point $z'\eqdef x_r$ is in $Z_t$. We then have  $\di(z_t,z')= r \di(z_t,z)$ and 
    \begin{equation}\label{eq:}
        g(z') \leq (1-r) g(z_t) + r g(z) - \frac{\lambda}{2} r (1-r) \di(z_t,z)^2,
    \end{equation}
    ensuring in particular that $z' \in \mathcal{D}[g]$. Observe that, since $z'$ lies on a geodesic connecting $z_t$ and $z$, we have that $\di(z_t,z)= \di(z_t,z') + \di(z',z)$. Moreover, the triangular inequality guarantees that  $\di(z_s,z)\leq \di(z_s,z') + \di(z',z)$ for every $s\in (0,T)$, therefore we deduce that
    \begin{equation}
    \begin{split}
        \frac{\de^+}{\de t} \frac{\di(z_t,z)^2}{2} = \di(z_t,z) \, \frac{\de^+}{\de t} \di(z_t,z) &\leq \di(z_t,z)\,\frac{\de^+}{\de t} \di(z_t,z') =   \dfrac{\di(z_t,z')}{r}\,\frac{\de^+}{\de t} \di(z_t,z') =  \frac{1}{r} \frac{\de^+}{\de t} \frac{\di(z_t,z')^2}{2}.
    \end{split}
    \end{equation}
    Then, using the assumption on the validity of \eqref{eq:EVIlambda} for any point in $Z_t \cap \mathcal{D}[g]$ and the fact that $z'\in Z_t \cap \mathcal{D}[g]$, we obtain 
    \begin{equation}
    \begin{split}
        \frac{\de^+}{\de t} \frac{\di(z_t,z)^2}{2} &\leq \frac 1r \bigg[ g(z')-g(z_t) - \frac\lambda 2 \di(z_t,z')^2 \bigg] \\
        &\leq \frac 1r \bigg[ (1-r) g(z_t) + r g(z) - \frac{\lambda}{2} r (1-r) \di(
        z_t,z)^2 -g(z_t) - \frac\lambda 2 \di(z_t,z')^2 \bigg] \\
        &= g(z)-g(z_t) - \frac\lambda 2 \di(z_t,z)^2,
    \end{split}
    \end{equation}
    concluding the proof.
\end{proof}

\begin{definition}[$\evikn$ gradient curve]
    Fix $K \in \R$ and $N<0$, and let $f\colon X \to \bar \R$ be a lower semicontinuous functional. 
    An \emph{$\evikn$ gradient curve} for~$f$ is a continuous curve $(y_t)_{t\in (0,T)} \subset \mathcal D[f]$, with $T\in (0,+\infty]$, such that 
    \begin{equation}\label{eq:EVIKN}
        \frac{\de^+}{\de t} \bigg[ \skn \bigg( \frac{\di(y_t,z)}{2}\bigg)^2\bigg] + K \skn \bigg( \frac{\di(y_t,z)}{2}\bigg)^2 \leq \frac{N}{2} \bigg( 1 - \frac{f_N(z)}{f_N(y_t)}\bigg),
    \end{equation}
     for every $z\in \bar{\mathcal D}[f]$ and every $t\in(0,T)$, with $\di (y_t,z)< \pi \sqrt{ K/N}$ if $K<0$.
\end{definition}

Through standard computations, equation \eqref{eq:EVIKN} can be rewritten to obtain the following equivalent formulations for $\evikn$ gradient curves, cf. equations 3.5 and 3.6 in \cite{ohta}.

\begin{prop}[Equivalent formulations for $\evikn$]\label{prop:equivalentEVIKN}
    The inequality \eqref{eq:EVIKN} defining an $\evikn$ gradient curve is equivalent to each of the following:
    \begin{enumerate}[$(i)$]
        \item\label{i:p:equivalentEVIKN:1} $\displaystyle \frac{\de^+}{\de t} \frac{\di(y_t,z)^2}{2} \leq \frac{N \, \di(y_t,z)}{ \skn \big(\di(y_t,z)\big)} \braket{ 1 - \frac{f_N(z)}{f_N(y_t)} } - 2 K \di(y_t,z) \frac{\skn \big(\di(y_t,z)/2\big)^2}{\skn \big(\di(y_t,z)\big)},$ 
        \item\label{i:p:equivalentEVIKN:2} $\displaystyle \frac{\de^+}{\de t} \frac{\di(y_t,z)^2}{2}  \leq \frac{N \, \di(y_t,z)}{ \skn \big(\di(y_t,z)\big)} \braket{ \ckn\big(\di(y_t,z)
        \big) - \frac{f_N(z)}{f_N(y_t)} }.$ 
    \end{enumerate}
\end{prop}

The following proposition shows the monotonicity in $K$ and $N$ of the notion of $\evikn$ gradient curves and was proved in \cite{ohta}, see Lemma 3.3 therein.

\begin{prop}[Monotonicity of $\evi$]\label{prop:monotonicity}
    If $(y_t)_{t\in (0,T)}$ is an $\evikn$ gradient curve for $f$, then it is an $\evi_{K',N'}$ gradient curve for $f$, for every $K'\leq K$ and $N'\in [N,0)$.
    Moreover, if $(y_t)_{t\in (0,T)}$ is an $\evi_K$ gradient curve for $f$, then it is a $\evi_{K,N}$ gradient curve for $f$, for every $N<0$.
\end{prop}

We conclude this section by proving that $f$ is non-increasing along every $\evikn$ curve and in particular this implies that no $\evikn$ curve exists starting at $y_0 \notin \mathcal D^\star[f]$.
 We remark that this foundational result was neither proved nor even observed before and will allow us to improve several results in \cite{ohta}.

\begin{prop}\label{prop:fdecreases}
    Let~$f\colon X \to \bar \R$ be a lower semicontinuous functional. Let $(y_t)_{t\in (0,T)}$ be an $\evikn$ gradient curve for $f$, for some $K \in \R$ and $N<0$. Then, the function 
    \begin{equation}\label{eq:increasing}
        (0, T)\ni t \mapsto h(t)\eqdef f(y_t)
    \end{equation}
    is non-increasing.
\end{prop}

\begin{proof}

Since the function $f$ is lower semicontinuous and $(y_t)_{t\in (0,T)}$ is continuous, the function~$h\colon (0,T)\to\bar\R$ too is lower semicontinuous.
Therefore, for every~$t_0,t_1\in(0,T)$ with $t_0<t_1$, we find that ~$\argmin_{[t_0,t_1]} h\neq \varnothing$, allowing us to select~$\bar t \in [t_0, t_1]$ where the minimum is attained.
    Now, for every $t\in[\bar t,t_1]$, letting~$z=y_{\bar t}$ in the $\evikn$ inequality, we get
     \begin{equation}
         \frac{\de^+}{\de t} \bigg[ \skn \bigg( \frac{\di(y_t,y_{\bar t})}{2}\bigg)^2\bigg] \leq - K \skn \bigg( \frac{\di(y_t,y_{\bar t})}{2}\bigg)^2 + \frac{N}{2} \braket{ 1 - \frac{f_N(y_{\bar t})}{f_N(y_t)} } \leq - K \skn \bigg( \frac{\di(y_t,y_{\bar t})}{2}\bigg)^2.
     \end{equation}
     As a consequence, Gr\"onwall's lemma on~$[\bar t,t_1]$ ensures that 
     \begin{equation}
         \skn \bigg( \frac{\di(y_t,y_{\bar t})}{2}\bigg)^2 \leq 0, \qquad  t\in[\bar t,t_1].
     \end{equation}
     From this we deduce that $y_t=y_{\bar t}$ for every $t\in [\bar t,t_1]$.
     Thus, ~$t_1\in \argmin_{[t_0,t_1]} h$, and the conclusion follows by arbitrariness of $t_0$ and $t_1$.
\end{proof}

\section{Reparametrization maps and $\evi$}\label{sec:3}
Everywhere in this section we fix a metric space $(X, \di)$, a \emph{densely defined} (i.e.~$\overline{\mathcal{D}[f]}=X$) \emph{lower semicontinuous} functional $f \colon X \to \bar \R$ and two constants, $K\in \R$ and $N<0$. 

\subsection{Definition of reparametrization maps}

Let us introduce the reparametrization maps, which will serve as a fundamental tool for studying the gradient flows of $(K,N)$-convex functionals. As already mentioned in the introduction, the aim of these reparametrization maps is to obtain a gradient curve for~$f_N$, starting from a gradient curve for~$f$, in order to exploit the improved convexity properties of $f_N$, when $f$ is $(K,N)$-convex. This idea will be formalized in Proposition \ref{prop:correspondence} and Proposition \ref{prop:correspondence2}. \medskip

We denote by $\mathcal{C}$ the set of all the continuous curves  parametrized on some interval $(0,T)$, with $T\in(0,+\infty]$, and with image in $\mathcal D[f]$. We introduce the following two subsets of $\mathcal{C}$:
\begin{equation}\label{eq:CPrimo}
    \mathcal{C}':= \Big\{(y_t)_{t\in (0,T)} \in \mathcal C\, :\, (0,T) \ni t \mapsto f (y_t) \text{ is non-increasing} \Big\},
\end{equation}
\begin{equation}
    \mathcal{C}''_N:= \Big\{(y_t)_{t\in (0,T)} \in \mathcal C '\, :\, \int_0^{T\wedge 1} f_N(y_t) \de t <\infty\Big\}.
\end{equation}

We now proceed to define the two reparametrization maps
\begin{equation}
    \rr_1\colon \mathcal{C}'  \longrightarrow \mathcal{C}''_N \qquad \text{and} \qquad \rr_2\colon \mathcal{C}''_N \longrightarrow \mathcal{C}'.
\end{equation}
The maps~$\mathcal R_1$ and~$\mathcal R_2$ depend on~$N$, though this dependence is omitted for simplicity of the notation.
In order to define $\rr_1$, take any curve $(y_t)_{t\in (0,T)}\in \mathcal{C}'$ and consider the function $\alpha\colon (0,T) \to (0,T')$ defined as 
    \begin{equation}\label{eq:AlphaReparametrization}
        \alpha(t)\eqdef  - N \int_0^t \frac{1}{f_N(y_r)}\de r \quad \text{for } t\in (0,T),
    \end{equation}
    where we set $T'\eqdef \lim_{t\to T^-}\alpha(t)$. Observe that $\alpha$ is  locally bi-Lipschitz and almost everywhere differentiable with derivative
    \begin{equation}
        \alpha'(t)= - \frac{N}{f_N(y_t)} >0.
    \end{equation}
    In particular, $\alpha$ is strictly increasing and thus invertible. Let $\varphi\colon (0,T') \to (0,T)$ be its inverse, for which we have $\lim_{s\downarrow 0}\varphi(s)=0$, $\lim_{s\to T'^-}\varphi(s)=T$ and  
    \begin{equation}\label{eq:phi}
        \varphi'(s)= \frac{1}{\alpha'(\varphi(s))} = - \frac 1N f_N(y_{\varphi(s)})>0, \qquad  \text{for a.e. } s\in [0,T'],
    \end{equation}
    so that also $\varphi$ is  locally bi-Lipschitz and strictly increasing. Finally, we set 
    \begin{equation}
        \rr_1\big((y_t)_{t\in (0,T)}\big) = (z_s)_{s\in (0,T')}\eqdef (y_{\varphi(s)})_{s\in (0,T')}.
    \end{equation}
    Observe that, given any $r\in (0, T' \wedge 1)$, the function $\varphi$ is Lipschitz on $[r, T '\wedge 1)$ and therefore $\int_r^{T' \wedge 1} \varphi'(s) \de s =\varphi(T' \wedge 1)- \varphi(r)$. Taking the limit as $r\to 0$, we deduce that 
    \begin{equation}
        \int_0^{T' \wedge 1}  f_N(y_{\varphi(s)}) \de s = - N \int_0^{T' \wedge 1} \varphi'(s) \de s = - N \lim_{s\to T' \wedge 1}\varphi(s) <\infty,
    \end{equation}
    showing in particular that $(z_s)_{s\in (0,T')} \in \mathcal C''_N$.
    
    In a similar way, we can define the reparametrization map $\rr_2$.
    We consider any $(z_s)_{s\in (0,T')}\in \mathcal{C}''_N$ and define the function $\beta\colon (0,T') \to (0,T)$ as 
    \begin{equation}
        \beta(s)\eqdef  -\frac 1N \int_0^s f_N(z_r)\de r,
    \end{equation}
    where we set $T=\lim_{s\to T'}\beta(s)$. Observe that, since $(z_s)_{s\in (0,T')}\in \mathcal{C}''_N$, $\beta$ is well defined. Moreover,
    it is  locally bi-Lipschitz and almost everywhere differentiable with derivative
    \begin{equation}
        \beta'(s)= -\frac{f_N(z_s)}{N} >0, \qquad \forall s \in (0,T'),
    \end{equation}
    which ensures that it is strictly increasing and thus invertible. Let $\psi\colon (0,T) \to (0,T')$ be its inverse, for which we have $\lim_{t \downarrow 0}\psi(t)=0$, $\lim_{t\to T^-}\psi(t)=T'$ and  
    \begin{equation}\label{eq:psi}
        \psi'(t)= \frac{1}{\beta'(\psi(t))} = - \frac N{f_N(z_{\psi(t)})} >0,\qquad  \text{for a.e. } t\in (0,T),
    \end{equation}
    in particular $\psi$ is locally bi-Lipschitz and strictly increasing. 
    Then, we set
    \begin{equation}
        \rr_2\big((z_s)_{s\in (0,T')}\big)=(y_t)_{t\in (0,T)}\eqdef (z_{\psi(t)})_{t\in (0,T)}.
    \end{equation}

\begin{prop}\label{prop:inverse}
    The maps $\rr_1$ and $\rr_2$ are inverse of each other, i.e. 
    \[\rr_1 \circ \rr_2=\id_{\mathcal{C}'} \, \text{ and } \, \rr_2 \circ \rr_1=\id_{\mathcal{C}''_N}.\]
\end{prop}

\begin{proof}
    Consider any $(y_t)_{t\in (0,T)}\in \mathcal{C}'$ and take $(z_s)_{s\in (0,T')}=\rr_1\big((y_t)_{t\in (0,T)}\big)$. According to the discussion above, the maps $\varphi$ and $\psi$ associated to these curves (respectively), are locally bi-Lipschitz. Thus, their composition $\varphi\circ\psi:(0,T) \to (0,T)$ is itself locally bi-Lipschitz and, by \eqref{eq:phi} and \eqref{eq:psi}, it satisfies
    \begin{equation}
        (\varphi\circ\psi)'(t) = \varphi'(\psi(t)) \,\psi'(t)= f_N(y_{\varphi\circ\psi (t)}) \, f_N(z_{\psi (t)})^{-1} = 1,  \qquad\text{for a.e. } t\in (0,T).
    \end{equation}
    Recalling that $\lim_{s\downarrow 0}\varphi(s)=0$ and $\lim_{t \downarrow 0}\psi(t)=0$, we deduce that $\varphi\circ\psi (t)=t$ for every $t\in (0,T)$ and, as a consequence, we conclude that $\rr_2\big(\rr_1\big((y_t)_{t\in (0,T)}\big)\big)=(y_t)_{t\in (0,T)}$. Therefore, we have shown that $\rr_1 \circ \rr_2=\id_{\mathcal{C}'}$, analogously we can prove that $\rr_2 \circ \rr_1=\id_{\mathcal{C}''_N}$.
\end{proof}

In the following, we denote by~$\EVIkn(f)$ the set of all $\evikn$ gradient curves for~$f$, parametrized on some interval~$(0,T)$, with~$T\in(0,+\infty]$. Similarly, we denote by~$\EVI_\lambda(f_N)$ the set of all $\evi_\lambda$ gradient curves for $f_N$, contained in $\mathcal{D}[f]$ and parametrized on some interval $(0,T')$, with $T'\in(0,+\infty]$.
In particular, according to Proposition~\ref{prop:fdecreases}, we have that $\EVIkn(f), \EVI_\lambda(f_N) \subset \mathcal{C}'$. 

\begin{prop}\label{prop:correspondence}
    It holds that
    \[
    \rr_1\tparen{\EVI_{0,N}(f)}= \EVI_0(f_N)\cap \mathcal C''_N \quad \text{and} \quad \rr_2(\EVI_0(f_N)\cap \mathcal C''_N)= \EVI_{0,N}(f).
    \]
    In particular, $\rr_1|_{\EVI_{0,N}(f)}$ and $\rr_2|_{\EVI_0(f_N) \cap \mathcal C''_N}$ are bijections from $\EVI_{0,N}(f)$ and $\EVI_0(f_N)\cap \mathcal C''_N$ and vice versa.
\end{prop}

\begin{proof}
    Take any $(y_t)_{t\in (0,T)}\in \EVI_{0,N}(f) (\subset \mathcal C')$  and consider 
    $(z_s)_{s\in (0,T')}=\rr_1\big((y_t)_{t\in (0,T)}\big)=(y_{\varphi(s)})_{s\in (0,T')}$. Then, given any $z\in \mathcal D[f_N]= \bar{\mathcal D}[f]$ and $s\in(0,T')$, we compute that 
    \begin{equation}\label{eq:changecorr}
    \begin{split}
         \frac{\de^+}{\de s} \frac{\di(z_s,z)^2}{2} &= 2 \frac{\de^+}{\de t}\bigg|_{t=\varphi(s)} \bigg[ \bigg( \frac{\di(y_t,z)}{2}\bigg)^2\bigg] \cdot \varphi'(t)
         \\
         &\leq N \braket{ 1 - \frac{f_N(z)}{f_N(z_s)}} \braket{ - \frac 1N f_N(z_s) }= f_N(z)-f_N(z_s),
    \end{split}
    \end{equation}
    where the inequality follows from the assumption that $(y_t)_{t\in (0,T)}\in \EVI_{0,N}(f)$.
    We have then shown that $\rr_1\tparen{\EVI_{0,N}(f)}\subset \EVI_0(f_N)\cap \mathcal C''_N$.
    
    On the other hand, take any $(z_s)_{s\in (0,T')}\in \EVI_0(f_N)\cap\mathcal C''_N$ and consider $(y_t)_{t\in (0,T)}=\rr_2\big((z_s)_{s\in (0,T')}\big)=(z_{\psi(t)})_{t\in (0,T)}$.
    Then, given any $z\in \bar{\mathcal D}[f]= \mathcal D[f_N]$ and $t\in(0,T)$, we compute that
    \begin{equation}\label{eq:changecorr2}
    \begin{split}
        \frac{\de^+}{\de t} \bigg[\bigg( \frac{\di(y_t,z)}{2}\bigg)^2\bigg] &= \frac 12 \frac{\de^+}{\de t}\bigg|_{s=\psi(t)} \bigg( \frac{\di(z_s,z)}{2}\bigg) \cdot  \psi'(t) \\
        &\leq [f_N(z) - f_N(z_s)] \frac{-N}{f_N(z_s)} = \frac{N}{2} \braket{ 1 - \frac{f_N(z)}{f_N(y_t)} },
    \end{split}
    \end{equation}
    where the inequality follows from the assumption that $(z_s)_{s\in (0,T')}\in \EVI_0(f_N)$. We have thus showed that $\rr_2\tparen{\EVI_0(f_N)\cap \mathcal C''_N}\subset \EVI_{0,N}(f)$.

    Finally, the inclusions $\rr_1\tparen{\EVI_{0,N}(f)}\subset \EVI_0(f_N)\cap \mathcal C''_N$ and $\rr_2(\EVI_0(f_N)\cap \mathcal C''_N)\subset\EVI_{0,N}(f) $, combined with Proposition \ref{prop:inverse}, are sufficient to conclude the proof.
\end{proof}

\subsection{Local characterization and self-improvement of $\evikn$ gradient curves}

In this section we prove a local characterization of $\evikn$ gradient curves, which is the analogous of Proposition 3.11 in \cite{MurSav}. This result is fundamental to prove the self-improvement property for $\evikn$ gradient curves, see Corollary \ref{cor:selfimpr}.

The local characterization relies on the following definitions.
    
\begin{definition}
    Let $g\colon X \to \bar \R$ be a functional on $X$ and  $(x_t)_{t\in [0,T]}$ be a geodesic starting at $x_0\in \mathcal{D}[g]$. We define the directional derivative as
    \begin{equation}
        g'(x_0;x)\eqdef  \liminf_{t\downarrow 0} \frac{g(x_t)-g(x_0)}{t}.
    \end{equation}
\end{definition}

\begin{definition}
    Let $(y_t)_{t\in (0,T)}$ be a continuous curve and $(z_t)_{t\in (0,1)}$ be the geodesic emanating from $z_0=y_{t_0}$, for some fixed $t_0\in (0,T)$. We define the quantity
    \begin{equation}\label{eq:def[]}
    \begin{split}
         [\dot y,z]_{t_0} \eqdef  \lim_{s \downarrow 0} \frac{1}{2s} \frac{\de^+}{\de t} \bigg|_{t=t_0} \di^2 (y_t,z_s) &= \sup_{0< s\leq 1} \frac{1}{2s} \frac{\de^+}{\de t}  \bigg|_{t=t_0}\di^2 (y_t,z_s)\\
         &= \sup_{0< s\leq 1} \limsup_{t\downarrow t_0} \frac{\di^2(y_t,z_s)-\di^2(y_{t_0},z_s)}{2s (t-t_0)},
    \end{split}
    \end{equation}
    where the last equality is a consequence of Lemma 3.13 in \cite{MurSav}.
\end{definition}

\begin{lemma}\label{lem:chain}
    Fix~$x_0\in \mathcal{D}[f]$. Then, 
    \begin{equation}
        f_N'(x_0;x) = - \frac{1}{N} f_N(x_0) f'(x_0;x), \qquad x\in X. \qedhere
    \end{equation}
\end{lemma}

\begin{proof}
    Recall that $f$ is lower semicontinuous. If $\limsup_{t\to 0} f(x_t) > f(x_0)$, then both the right-hand side and the left-hand side are equal to $+\infty$. Otherwise, we have that $\limsup_{t\to 0} f(x_t) = \liminf_{t\to 0} f(x_t) = f(x_0)$ and to conclude the proof it is sufficient to observe that 
    \[
        f_N(x_t)-f_N(x_0) = - \frac{1}{N} f_N(x_0) \tbraket{ f(x_t)-f(x_0) } + o \tparen{f(x_t)-f(x_0)}. \qedhere
    \]
\end{proof}

\begin{lemma}\label{lem:KNconv}
    Let $f$ be $(K,N)$-convex and $(x_s)_{s\in [0,1]}$ be a geodesic such that $x_0\in \mathcal{D}[f]$ and $\di\eqdef  \di(x_0,x_1)<\pi \sqrt{ K/N}$ when $K<0$. Then, it holds that
    \begin{equation}
        f'_N(x_0;x) \leq \frac{\di}{\skn (\di)} \tbraket{f_N(x_1)- \ckn (\di) f_N(x_0)}.
    \end{equation}
\end{lemma}

\begin{proof}
Directly from the definition of directional derivative of $f_N$ and the $(K, N)$-convexity of $f$, we obtain the following inequality
\[
f'_N(x_0; x) \le \liminf_{t \downarrow 0} \dfrac{\sigma_{K, N}^{(t)}(\di) f_N(x_1) + \big(\sigma_{K, N}^{(1-t)}(\di) -1 \big) f_N(x_0)}{t}.
\]
Hence, the validity of the following identities
\[
\lim_{t \downarrow 0} \dfrac{\sigma_{K, N}^{(t)}(\di)}{t} = \dfrac{\di}{\skn(\di)} \quad \text{ and } \quad \lim_{t \downarrow 0} \dfrac{\sigma_{K, N}^{(1-t)}(\di) - 1}{t} = - \dfrac{\ckn(\di)}{\skn(\di)} \di
\]
leads to the desired conclusion.
\end{proof}

\begin{prop}
     If $(y_t)_{t\in (0,T)}$ is a $\evikn$ gradient curve for $f$, then, for every $t\in (0,T)$ and every geodesic $(z_s)_{s\in [0,1]}$ emanating from $z_0=y_t$, with $\di(z_0,z_1)<\pi \sqrt{ K/N}$ when $K<0$, it holds that
     \begin{equation}\label{eq:diffchar}
         [\dot y, z]_t\leq f'(y_t;z).
     \end{equation}
     Conversely, if $(y_t)_{t\in (0,T)}\in \mathcal C$ satisfies \eqref{eq:diffchar} and $f$ is $(K,N)$-convex, then $(y_t)_{t\in (0,T)}$ is an $\evikn$ gradient curve for $f$.
\end{prop}

\begin{proof}
    To prove the first part of the statement, we fix $t\in (0,T)$ and a geodesic $(z_s)_{s\in [0,1]}$ emanating from $z_0=y_t$, with $\di(z_0,z_1)<\pi \sqrt{ K/N}$ if $K<0$. As $(y_t)_{t\in (0,T)}$ is an $\evikn$ gradient curve for $f$, from item $(i)$ Proposition \ref{prop:equivalentEVIKN}, applied with $z=z_s$, we know that
    \begin{equation}
    \begin{split}
        \frac{\de^+}{\de t} \frac{\di(y_t,z_s)^2}{2}  &\leq \frac{N \, \di(y_t,z_s)}{ \skn \big(\di(y_t,z_s)\big)} \braket{ 1 - \frac{f_N(z_s)}{f_N(y_t)} } - 2 K \di(y_t,z_s) \frac{\skn \big(\di(y_t,z_s)/2\big)^2}{\skn \big(\di(y_t,z_s)\big)}
        \\
         &=\frac{s N \di}{ \skn (s\, \di)} \braket{ 1 - \frac{f_N(z_s)}{f_N(y_t)} } - 2 s K \di \frac{\skn(s\di/2)^2}{\skn (s\, \di)},
    \end{split}
    \end{equation}
    where we denoted $\di\eqdef \di(y_t,z_1)$ and used that $(z_s)_{s\in [0,1]}$ has constant speed. Recalling \eqref{eq:def[]} and that 
    \begin{equation}
        \skn (\theta) = \theta + o (\theta) \qquad \text{as } \theta \to 0,
    \end{equation}
    we deduce that 
    \begin{equation}
        [\dot y, z]_t\leq \liminf_{s\downarrow 0} \frac{s N \di}{ \skn (s\di)} \braket{ 1 - \frac{f_N(z_s)}{f_N(y_t)} } = - N \frac{1}{f_N(y_t)} \liminf_{s\downarrow 0}  \frac{f_N(z_s)-f_N(y_t)}{s}= f'(y_t;z),
    \end{equation}
    where the last equality is a consequence of Lemma \ref{lem:chain}.

    To prove the second part of the statement, we fix a time $t\in(0,T)$ and any $z\in \mathcal{D}[f]$, with $\di(y_t,z)<\pi \sqrt{ K/N}$ if $K<0$. Let $(z_s)_{s\in [0,1]}$ be a geodesic connecting $z_0=y_t$ and $z_1=z$, then using \eqref{eq:def[]}, our assumption \eqref{eq:diffchar} and Lemmas \ref{lem:chain} and \ref{lem:KNconv} we deduce that
    \begin{equation}
    \begin{split}
        \frac{\de^+}{\de t} \frac{\di(y_t,z)^2}{2} \leq [\dot y,z]_t \leq f'(y_t;z) = -N &\frac{1}{f_N(y_t)} f'_N(y_t;z) \\
        &\leq \frac{N \, \di(y_t,z)}{ \skn \big(\di(y_t,z)\big)} \braket{ \ckn\big(\di(y_t,z)
        \big) - \frac{f_N(z)}{f_N(y_t)} }.
    \end{split}
    \end{equation}
    The thesis follows from Proposition~\ref{prop:equivalentEVIKN}\ref{i:p:equivalentEVIKN:2}.
\end{proof}

\begin{corollary}[Self-improvement of $\evi$]\label{cor:selfimpr}
    Let $K, K' \in \R$ and $N,N'\in (-\infty, 0)$ be such that $K\leq K'$ and $N\geq N'$. 
    If $(y_t)_{t\in (0,T)}$ is a $\evikn$ gradient curve for $f$ and $f$ is $(K',N')$-convex, then $(y_t)_{t\in (0,T)}$ is a $\evi_{K',N'}$ gradient curve for $f$.
\end{corollary}

\begin{remark}
    As an interesting consequence of the previous result, combined with Proposition \ref{prop:correspondence} we obtain that, if $f$ is $(K,N)$-convex for some $K\geq 0$ and $N\in (-\infty, 0)$, then 
    \begin{equation}
        \rr_1\tparen{\EVI_{K,N}(f)}= \EVI_0(f_N)\cap \mathcal C''_N\qquad \text{and} \qquad\rr_2\tparen{\EVI_0(f_N)\cap \mathcal C''_N}= \EVI_{K,N}(f).
    \end{equation}
    This is consistent with Proposition \ref{prop:ftofN}\ref{i:p:ftofN:1}.
\end{remark}

Taking advantage of Corollary \ref{cor:selfimpr}, we can also prove an analogue of Proposition \ref{prop:correspondence} for negative values of $K$, assuming the functional $f$ to be $(K,N)$-convex and bounded. 

\begin{prop}\label{prop:correspondence2}
     If $f$ is bounded above by $M$ and $(K,N)$-convex for some $K< 0$ and $N<0$, then $\rr_1\tparen{\EVI_{K,N}(f)}\subset \EVI_{\tilde M}(f_N) \cap \mathcal C''_N$ and $\rr_2(\EVI_{\tilde M}(f_N)\cap \mathcal C''_N)\subset \EVI_{K,N}(f)$, where $\tilde M \eqdef  -\frac{K}{N} e^{-M/N}$.
     In particular, $\rr_1|_{\EVI_{K,N}(f) }$ and $\rr_2|_{\EVI(f_N)_{\tilde M} \cap \mathcal C''_N}$ are bijections from $\EVI_{K,N}(f) $ and $\EVI_{\tilde M}(f_N)\cap \mathcal C''_N$ and vice versa.
\end{prop}

\begin{proof}
    To prove the first part of the statement, we take any $(y_t)_{t\in (0,T)}\in \EVI_{K,N}(f)$ and consider $(z_s)_{s\in (0,T')}=\rr_1\tparen{(y_t)_{t\in (0,T)}}=(y_{\varphi(s)})_{s\in (0,T')}$, we want to prove that $(z_s)_{s\in (0,T')}\in \EVI_{\tilde M}(f_N)$.
    Given any $\varepsilon>0$, we consider $\bar \theta$ sufficiently small such that 
    \begin{equation}\label{eq:stima}
        \theta\frac{\skn(\theta / 2)^2}{\skn(\theta)} \leq \frac{1}{4}+ \varepsilon, \qquad \text{for every }\theta\leq \bar \theta.
    \end{equation}
    Then, given any $s\in(0,T')$ and $z\in B_{\bar \theta}(z_s) \subset \mathcal D[f_N] = \bar{\mathcal D}[f]= X$, proceeding as in the proof of Proposition~\ref{prop:correspondence} (in particular, see \eqref{eq:changecorr}) and using Proposition~\ref{prop:equivalentEVIKN}\ref{i:p:equivalentEVIKN:2}, we deduce that 
    \begin{equation}\label{eq:boh}
    \begin{split}
         \frac{\de^+}{\de s} \frac{\di(z_s,z) ^2}{2} \leq \frac{\di }{\skn(\di )} \braket{f_N(z)-f_N(z_s)} + 2 \frac{K}{N} f_N(z_s)\, \di \frac{\skn(\di  / 2)^2}{\skn(\di )} ,
    \end{split}
    \end{equation}
    where $\di=\di(z_s,z)$.
    Now, if $f_N(z)\leq f(z_s)$, using \eqref{eq:stima} and that $\skn(\theta)\leq \theta$, we have  
    \begin{equation}\label{eq:mah}
        \frac{\de^+}{\de s} \frac{\di(z_s,z) ^2}{2} \leq f_N(z)-f_N(z_s) +  2 \braket{\frac{1}{4}+ \varepsilon} \frac{K}{N} e^{-M/N}.
    \end{equation}
    While, if $f_N(z)\geq f(z_s)$, we observe that the right-hand side of \eqref{eq:boh} is equal to 
    \begin{equation}
    \begin{split}
          \braket{f_N(z)-f_N(z_s)} &\braket{ \frac{\di }{\skn(\di )}  - 2 \frac{K}{N} \di \frac{\skn(\di  / 2)^2}{\skn(\di )}} + 2 \frac{K}{N} f_N(z) \,\di \frac{\skn(\di  / 2)^2}{\skn(\di )} \\
          &=\braket{f_N(z)-f_N(z_s)} \braket{ \di \frac{\ckn(\di)}{\skn(\di )} } + 2 \frac{K}{N} f_N(z) \,\di \frac{\skn(\di  / 2)^2}{\skn(\di )}\\
          &\leq f_N(z)-f_N(z_s) + 2 \braket{\frac{1}{4}+ \varepsilon} \frac{K}{N} e^{-M/N},
    \end{split}
    \end{equation}
    showing that \eqref{eq:mah} holds for every $z\in B_{\bar \theta}(z_s)$. According to Proposition \ref{prop:globalization} combined with Proposition~\ref{prop:ftofN}\ref{i:p:ftofN:2}, this is sufficient to show that $(z_s)_{s\in (0,T')}\in \EVI_{(1+4\varepsilon)\tilde M}(f_N)$. The conclusion follows since $\varepsilon$ was arbitrary. 

    For the second part of the statement, take any $(z_s)_{s\in (0,T')}\in \EVI_{\tilde M}(f_N)\cap \mathcal C''_N$ consider $(y_t)_{t\in (0,T)}=\rr_2\big((z_s)_{s\in (0,T')}\big)=(z_{\psi(t)})_{t\in (0,T)}$.
    In particular, $(y_t)_{t\in (0,T)} \in \mathcal C'$ and therefore, given any $\bar T <T$, we have that
    \begin{equation}
        L:= \inf_{t\in (0,\bar T)} f(y_t) = \liminf_{t\to \bar T^-} f(y_t) \geq f(y_{\bar T}) >-\infty .
    \end{equation}
     Then, given any $t\in(0,\bar T)$ and $z\in  \mathcal D[f_N] = \bar{\mathcal D}[f]= X$, proceeding as in the proof of Proposition~\ref{prop:correspondence} (see in particular \eqref{eq:changecorr2}), we obtain that
    \begin{equation}\label{eq:noia}
    \begin{split}
        \frac{\de^+}{\de t} \frac{\di(y_t,z)^2}{2} \leq N \braket{1- \frac{f_N(z)} {f_N(y_t)}} + \frac{\tilde M N}{2} \frac{\di^2}{f_N(y_t)}.
    \end{split}
    \end{equation}
    Moreover, we observe that, for every $\theta$,
    \begin{equation}\label{eq:mmm}
        \theta \leq 4 \frac{\skn(\theta/2)^2}{\skn(\theta)}.
    \end{equation}
    Now, if $f_N(z)\geq f(y_t)$, using \eqref{eq:mmm} and that $\skn(\theta)\leq \theta$, we deduce  
    \begin{equation}
        \frac{\de^+}{\de t} \frac{\di(y_t,z)^2}{2} \leq \frac{N \, \di}{ \skn (\di)} \braket{ 1 - \frac{f_N(z)}{f_N(y_t)} } - 2 K \frac{e^{-M/N}}{e^{-L/N}}\, \di \frac{\skn (\di/2)^2}{\skn (\di)}.
    \end{equation}
    While, if $f_N(z)\leq f(y_t)$, the right-hand side of \eqref{eq:noia} is less than or equal to 
    \begin{equation}
    \begin{split}
        N \braket{1- \frac{f_N(z)} {f_N(z_s)} }& \di \frac{\ckn(\di)}{\skn(\di )}  - 2 K \frac{e^{-M/N}}{e^{-L/N}}\,  \di \frac{\skn (\di/2)^2}{\skn (\di)}
        \\
        &= N \braket{1- \frac{f_N(z)} {f_N(z_s)}} \braket{\frac{\di }{\skn(\di )}  - 2 \frac{K}{N} \di \frac{\skn(\di  / 2)^2}{\skn(\di )} } - 2 K \frac{e^{-M/N}}{e^{-L/N}}\, \di \frac{\skn (\di/2)^2}{\skn (\di)} 
        \\
        &= \frac{N \di }{\skn(\di )} \braket{1- \frac{f_N(z)} {f_N(z_s)}} - 2 \braket{1- \frac{f_N(z)} {f_N(z_s)}+\frac{e^{-M/N}}{e^{-L/N}} }  K \di \frac{\skn (\di/2)^2}{\skn (\di)}
        \\
        &\leq \frac{N \di }{\skn(\di )} \braket{1- \frac{f_N(z)} {f_N(z_s)} } - 2 \braket{1+\frac{e^{-M/N}}{e^{-L/N}} }  K \di \frac{\skn (\di/2)^2}{\skn (\di)}.
    \end{split}
    \end{equation}
    Therefore, according to Proposition~\ref{prop:equivalentEVIKN}\ref{i:p:equivalentEVIKN:1}, we have proved that $(y_t)_{t\in (0,\bar T)}$ is an $\evi_{\tilde K,N}$ gradient curve for $f$, with $\tilde K = K \braket{1+\frac{e^{-M/N}}{e^{-L/N}} }$. Finally, Corollary \ref{cor:selfimpr} allows to conclude that $(y_t)_{t\in (0,\bar T)} \in \EVI_{K,N}(f)$ and, as this is true for any $\bar T<T$, we obtain the thesis.
\end{proof}

\section{Fundamental properties of $\evikn$ gradient curves}\label{sec4}
As in the previous section, we fix a metric space $(X, \di)$, a \emph{densely defined} (i.e.~$\overline{\mathcal{D}[f]}=X$) \emph{lower semicontinuous} functional $f \colon X \to \bar \R$ and two constants, $K\in \R$ and $N<0$.

 The main aim of this section is to prove the uniqueness (assuming existence) of $\evikn$ gradient curves for $f$ starting from a point $y_0\in X$. As it will be clear from the proof of Theorem~\ref{thm:uniqueness} below, Proposition \ref{prop:correspondence} plays a fundamental role in proving uniqueness when $K\geq 0$. Instead, if $K<0$ and $f$ is not bounded from above, there is no reason for $f_N$ to be $\lambda$-convex for any $\lambda \in \R$. (Recall the bound \eqref{eq:KNdiff2}.)
 This in particular means that also the reparametrization of any $\evikn$ gradient curve for $f$ will not be an $\evi_\lambda$ gradient curve for $f_N$, for any $\lambda \in \R$. However, following verbatim the proof of Proposition \ref{prop:correspondence2}, we can deduce the following result, which does not even require the $(K,N)$-convexity of $f$:

\begin{prop}\label{prop:correspondence3}
      Let $(y_t)_{t\in (0,T)}$ be an $\evi_{K,N}$ gradient curve starting from $y_0\in \mathcal D [f]$, for some $K< 0$ and $N<0$. Then there exists $R>0$ such that for the curve $(z_s)_{s\in (0,T')}=\rr_1\big((y_t)_{t\in (0,T)}\big)$ and for every $M\geq f(y_0)$ it holds that 
     \begin{equation}\label{eq:EVIlambdasublevel}
        \frac{\de^+}{\de s} \frac{\di(z_s,z)^2}{2}   +  \tilde M \di(z_s,z)^2 \leq f_N(z)-f_N(z_s) \qquad s \in (0,T'), \quad z\in \{f\leq M\}\cap B_R(z_s),
    \end{equation}
    where $\tilde M \eqdef  -\frac{K}{N} e^{-M/N}$.
\end{prop}

\begin{theorem}\label{thm:uniqueness}
    Let let $f\colon X \to \bar \R$ be a lower semicontinuous functional and consider two parameters $K\in \R$ and  $N<0$. Let $(y_t)_{t\in (0,T)}$ and $(\tilde y_t)_{t\in (0,\tilde T)}$ be two $\evikn$ gradient curves  for $f$, and denote by $(z_s)_{s\in (0,T')}=\rr_1\big((y_t)_{t\in (0,T)}\big)$ and $(\tilde z_s)_{s\in (0,\tilde T')}=\rr_1\big((\tilde y_t)_{t\in (0,\tilde T)}\big)$ their respective reparametrizations. The following statements hold:
    \begin{enumerate}[$(i)$]
        \item\label{i:t:uniqueness:1} if $K\geq 0$, then \begin{equation}\label{eq:contraction1}
            \di(z_s,\tilde z_s)\leq \di(z_r,\tilde z_r) \qquad \text{for every }0<r\leq s< \min \{T', \tilde T'\}.
        \end{equation}
        \item\label{i:t:uniqueness:2} If $K< 0$ and $f$ is $(K,N)$-convex and bounded from above by some $M > 0$,      then \begin{equation}\label{eq:contraction2}
            \di(z_s,\tilde z_s)\leq e^{- \tilde M (s-r)} \di(z_r,\tilde z_r) \qquad \text{for every }0<r\leq s< \min \{T', \tilde T'\},
        \end{equation}
        where $\tilde M \eqdef  -\frac{K}{N} e^{-M/N}$.
        \item\label{i:t:uniqueness:3} If $K< 0$ and it holds $\di(z_r,\tilde z_r)< R/3$ for some $r\in (0, \min \{T', \tilde T'\})$, where $R$ is the constant defined in Proposition \ref{prop:correspondence3},  then there exists $\varepsilon>0$ such that \begin{equation}\label{eq:contraction3}
            \di(z_s,\tilde z_s)\leq e^{- \tilde M (s-r)} \di(z_r,\tilde z_r) \qquad \text{for every }r\leq s< r+\varepsilon,
        \end{equation}
        where $\tilde M \eqdef  -\frac{K}{N} e^{-M/N}$ and $M:=\max\{ f(z_r),f(\tilde z_r)\}$.
    \end{enumerate}
    Moreover, for every $y_0\in \mathcal D[f]$ there exists at most one $\evikn$ gradient curve for $f$ starting from $y_0$, both for $K < 0$ and for $K \ge 0$. In the case $K \ge 0$ the initial point $y_0$ can be taken from the entire space $X$. This is also true for $K <0$, provided that the functional $f$ is $(K,N)$-convex and bounded from above.
\end{theorem}

\begin{proof}
    \ref{i:t:uniqueness:1}
    According to Proposition~\ref{prop:correspondence}, both $(z_s)_{s\in (0,T')}$ and $(\tilde z_s)_{s\in (0,\tilde T')}$ are $\evi_0$ gradient curves for $f_N$, starting from $y_0$. Then, Proposition \ref{prop:EVI} immediately gives \eqref{eq:contraction1}.
    
    \medskip
    
    \noindent \ref{i:t:uniqueness:2} Using Proposition \ref{prop:correspondence2} instead of Proposition \ref{prop:correspondence}, we deduce \eqref{eq:contraction2} analogously.
    
    \medskip 
    
    \noindent\ref{i:t:uniqueness:3} Take $\varepsilon>0$ such that for every $s\in(r,r+\varepsilon)$ we have $\di(z_r,z_s),\di(\tilde z_r,\tilde z_s) < R/3$, then $\di(z_s,\tilde z_t) < R$ for every $s, t \in (r, r+\varepsilon)$. Moreover, according to Proposition \ref{prop:monotonicity}, we know that $f(z_s), f(\tilde z_s)\leq M$ for every $s\in (r, r+\varepsilon)$. In particular, using Proposition \ref{prop:correspondence3} and following the classical proof of the contractivity for $\evi_\lambda$ gradient flows (see \cite[Theorem 3.5]{MurSav}), we deduce the contractivity estimate \eqref{eq:contraction3}.

    \medskip 

    \noindent When $K\geq 0$, the uniqueness of $\evikn$ gradient curves starting from a point $y_0 \in X$ is a straightforward consequence of point \ref{i:t:uniqueness:1}. Analogously, from \ref{i:t:uniqueness:2} we deduce it when $y_0\in X$ if $K<0$ and $f$ is $(K,N)$-convex and bounded above. Finally, in the case $K<0$, without any additional assumption on the potential $f$, we can use point \ref{i:t:uniqueness:3} to prove uniqueness of $\evikn$ gradient curves starting from a point $y_0 \in \mathcal D[f]$.
\end{proof}

\begin{prop}\label{prop:prop}
    Let $f\colon X \to \bar \R$ be a lower semicontinuous functional. Given any $K\in \R$ and $N<0$, let $(y_t)_{t\in (0,T)}$ be an $\evikn$ gradient curve for $f$. Then the curve $(y_t)_{t\in (0,T)}$ is locally Lipschitz on~$(0,T)$.
\end{prop}

\begin{proof}
     When $K\geq 0$, this is a simple consequence of Proposition \ref{prop:correspondence}.
    In fact, it guarantees that $(z_s)_{s\in (0,T')}=\rr_1\big((y_t)_{t\in (0,T)}\big)$ is a $\evion$ gradient flow for $f_N$ and it is therefore a locally Lipschitz curve. Moreover, as $f_N$ is non-increasing along $(z_s)_{s\in (0,T')}$, from \eqref{eq:psi} we observe that the reparametrization $\rr_2$ on $(z_s)_{s\in (0,T')}$ is locally Lipschitz.
    Finally, we deduce that $(y_t)_{t\in (0,T)}=\rr_2\big((z_s)_{s\in (0,T')}\big)$ (cf. Proposition \ref{prop:inverse}) is locally Lipschitz.
    When instead $K<0$, Proposition \ref{prop:correspondence3} ensures that $(z_s)_{s\in (0,T')}$ satisfies \eqref{eq:EVIlambdasublevel}.
    Following verbatim the proof of Theorem 3.5 in \cite{MurSav} (see in particular ‘Regularizing effects'), it is easy to note that \eqref{eq:EVIlambdasublevel} is sufficient to deduce that $(z_s)_{s\in (0,T')}$ is locally Lipschitz.
    We can then conclude the thesis as before. 
\end{proof}

The next proposition shows that the existence of a $\evikn$ gradient flow forces the potential to be $(K,N)$-convex. The analogous result has been proved for $\evi_\lambda$ gradient flows in \cite[Thm. 3.2]{DanSav} and for $\evikn$ gradient flows for positive values of the dimensional parameter $N$ in \cite[Thm. 2.23]{ErbarKuwadaSturm}. Our proof will follow the strategy developed in \cite[Thm. 2.23]{ErbarKuwadaSturm}, adapted to the negative dimensional case. 

\begin{prop}\label{eviconv}
    Let $f\colon X \to \bar \R$ be a densely defined lower semicontinuous functional. If for every $y_0\in \{f>-\infty\}$ there exists an $\evikn$ gradient curve $(y_t)_{t\in (0,T)}$ for $f$ starting from $y_0$, then $f$ is $(K,N)$-convex.
\end{prop}

\begin{proof}
We start by dealing with the case $K\neq 0$. Take any two points $x_0,x_1\in \mathcal{D}^\star[f]$, with $\di\eqdef \di(x_0,x_1)<\pi\sqrt{N/K}$ when $K<0$, and consider any geodesic $\gamma\colon [0,1]\to X$ connecting $x_0$ and $x_1$. Given any $s\in(0,1)$, we want to prove the convexity inequality \eqref{eq:convinequality} for $\gamma$ at time $t=s$. Observe that if $f(\gamma_s)=-\infty$ the inequality is trivially verified, otherwise we consider the $\evikn$ gradient curve $(\gamma_s^t)_{t\in (0,T)}$ starting from $\gamma_s$. According to \cite[Prop. 3.6]{ohta}, we know that 
\begin{equation}\label{eq:integralversion}
    \frac{N(e^{Kt}-1)}{2K} \bigg( 1 - \frac{f_N(z)}{f_N(\gamma_s^t)}\bigg) \geq e^{Kt} \skn\bigg(\frac{\di(\gamma_s^t, z)}{2}\bigg)^2 - \skn\bigg(\frac{\di(\gamma_s, z)}{2}\bigg)^2
\end{equation}
for all $z\in {\mathcal{D}}[f]$ and $t\in (0,T)$ such that $\sup_{r\in(0,t)} \di(\gamma_s^r,z)<\pi\sqrt{N/K}$ when $K<0$. Then, taking advantage of the trigonometric identity
\begin{equation}
    \skn(\theta/2)^2= - \frac{N}{2K} \bigg(\ckn(\theta) -1\bigg),
\end{equation}
we sum \eqref{eq:integralversion} for $z=x_0$ multiplied by $\sigma_{K,N}^{(1-s)}(\di)$ and \eqref{eq:integralversion} for $z=x_1$ multiplied by $\sigma_{K,N}^{(s)}(\di)$, obtaining
\begin{equation}\label{eq:stima2}
\begin{split}
    \sigma_{K,N}^{(1-s)}(\di) f_N(x_0) &+ \sigma_{K,N}^{(s)}(\di) f_N(x_1) \\
    &\geq \frac{f_N(\gamma_s^t)}{(e^{Kt}-1)} \bigg[\sigma_{K,N}^{(1-s)}(\di)\big(e^{Kt} \ckn(\di(\gamma_s^t,x_0))- \ckn(s\di)\big)\\
    &\qquad \qquad\qquad\qquad+ \sigma_{K,N}^{(s)}(\di)\big(e^{Kt} \ckn(\di(\gamma_s^t,x_1))- \ckn( (1-s)\di)\big) \bigg],
\end{split}
\end{equation}
for every $t$ sufficiently small. Now, as a consequence of the trigonometric sum formulas, we observe that 
\begin{equation}
    \sigma_{K,N}^{(1-s)}(\di) \ckn(s\di)+ \sigma_{K,N}^{(s)}(\di) \ckn((1-s)\di) = 1.
\end{equation}
 Moreover, we take adventage of the following trigonometric inequalities stated at the end of the proof of \cite[Thm. 2.23]{ErbarKuwadaSturm}: for every $\alpha,\alpha'\geq0$ and $\varepsilon,\varepsilon' \in [-\pi,\pi]$ such that $\varepsilon+\varepsilon'\geq 0$, setting $\beta= \alpha+\varepsilon$ and $\beta'=\alpha'+\varepsilon'$, we have that
\begin{equation}
    \begin{split}
        \sin(\alpha) \cos(\beta') + \cos(\beta) \sin(\alpha') &\leq \sin(\alpha + \alpha' ),\\
    \sinh(\alpha) \cosh(\beta') + \cosh(\beta) \sinh(\alpha') &\geq \sinh(\alpha + \alpha').
    \end{split}
\end{equation}
Applying these with $\alpha = (1-s) \di$, $\alpha'= s \di$, $\varepsilon= \di(\gamma_s^t,x_1)-(1-s)\di$ and $\varepsilon'= \di(\gamma_s^t,x_0)-s \di$, in which case $\varepsilon+\varepsilon'\geq 0$ by the triangular inequality, we deduce that
\begin{equation}
    \sigma_{K,N}^{(1-s)}(\di) \ckn(\di(\gamma_s^t,x_0))+ \sigma_{K,N}^{(s)}(\di) \ckn(\di(\gamma_s^t,x_1))\, \begin{cases}
        \leq 1 &\text{if }K<0,\\
        \geq 1 &\text{if }K>0.
    \end{cases}
\end{equation}
As a consequence, we obtain that the right-hand side of \eqref{eq:stima2} is greater than or equal to $f_N(\gamma_s^t)$, for every $t$ sufficiently small. The convexity inequality \eqref{eq:convinequality} follows by taking the limit as $t\to 0$.

To deal with the case $K=0$, it is sufficient to observe that any $\evi_{0,N}$ gradient curve for $f$ is also an $\evikn$ gradient curve for $f$, for every $K<0$. Then, according to the first part of the proof, the functional $f$ is $(K,N)$-convex for every $K<0$, thus it is also $(0,N)$-convex.
\end{proof}

\section{$\evi$ and $\edi$}\label{sec5}
Everywhere in this section we fix a metric space $(X, \di)$, a \emph{densely defined} (i.e.~$\overline{\mathcal{D}[f]}=X$) \emph{lower semicontinuous} functional $f \colon X \to \bar \R$ and two constants, $K\in \R$ and $N<0$.

\begin{definition}[Descending slope]
    For every $y\in \mathcal D [f]$, we define the descending slope $|D^- f|(y)$ as 
    \begin{equation}
    |D^- f|(y) \eqdef  \limsup_{z\to y} \frac{\big(f(z)-f(y)\big)^-}{\di(z,y)} .
\end{equation}
\end{definition}

We recall that we say that a curve~$(y_t)_{t\in (0,T)}$ \emph{starts at~$y_0\in X$} if there exists~$\lim_{t\downarrow 0} y_t=y_0$.
\begin{definition}[$\edi$ and curves of maximal slope]\label{d:EDI}
We say that a curve~$(y_t)_{t\in (0,T)} \subset \mathcal D[f]$ starting at $y_0$ satisfies the \emph{Energy-Dissipation Inequality} \eqref{eq:edi}, if $y\in \mathrm{AC}_{\loc}^2\big((0,T), X\big)$ and 
    \begin{equation}\label{eq:edi}\tag{EDI}
        f(y_t) \leq \lim_{\epsilon\downarrow 0}f(y_\epsilon) - \int_0^t \braket{ \frac{1}{2} |\dot{y}_s|^2 + \frac{1}{2} |D^- f|^2 (y_s) } \de s \qquad \text{for every }t\in (0,T).
    \end{equation}
We further say that a curve~$(y_t)_{t\in (0,T)} \subset \mathcal D[f]$ is \emph{of maximal slope for~$f$} if
\begin{enumerate}[$(a)$]
\item the curve $y$ is in $\mathrm{AC}^1_{\loc}\big((0,T),X\big)$;
\item\label{i:d:EDI:2} the function $t\mapsto f(y_t)$ is in $\mathrm{AC}^1_{\loc}\big((0,T),\R\big)$;
\item the following \emph{pointwise} \emph{Energy Dissipation Inequality} holds
   \begin{equation}\label{eq:maxslope}
        -\frac{\de}{\de t} f(y_t) \geq  \frac{1}{2} |\dot{y}_t|^2 + \frac{1}{2} |D^- f|^2 (y_t) \qquad \text{for }\Leb^1\text{-a.e. } t\in (0,T).  
    \end{equation}
\end{enumerate}
\end{definition}
Using the definition of descending slope $|D^-f|$, it is readily seen that \eqref{eq:maxslope} implies 
\begin{equation}\label{eq:equality}
    \frac{\de}{\de t} f(y_t) = - |D^- f|^2 (y_t) = -|\dot{y}_t|^2  \qquad \text{for }\Leb^1\text{-a.e. } t\in (0,T).
\end{equation}
In turn, for every curve of maximal slope~$(y_t)_{t\in (0,T)}$ starting at~$y_0$,
\eqref{eq:equality} is equivalent to the so-called \emph{Energy Dissipation Equality} \eqref{eq:ede}
\begin{equation}\label{eq:ede}\tag{EDE}
        f(y_t) = \lim_{\epsilon\downarrow 0}f(y_\epsilon) - \int_0^t \braket{ \frac{1}{2} |\dot{y}_s|^2 + \frac{1}{2} |D^- f|^2 (y_s) } \de s  \qquad \text{for every }0< t<T.
\end{equation}

\begin{remark}[Cf.~{\cite[Rmk.~2.6]{LzDSMaaPed23}}]
Definition~\ref{d:EDI} is slightly more restrictive than other definitions of curve of maximal slope in the literature, cf., e.g.~\cite{DeGMaTo}, in that we additionally assume~\ref{i:d:EDI:2} above.
Among other things, this grants that:~$\abs{D^-f}$ is a strong upper gradient for~$f$ along~$(y_t)_t$, cf. e.g.~\cite[Rmk.~2.8]{AGSCalc}; and that~$t\mapsto f(y_t)$ is non-increasing, by integrating~\eqref{eq:ede}.
\end{remark}

\begin{prop}\label{p:EVI-MaxSlope}
    Let $(y_t)_{t\in (0,T)}$ be an $\evikn$ gradient curve for~$f$, for some $K\in \R$ and $N<0$.
    Then,~$(y_t)_{t\in (0,T)}$ is a curve of maximal slope for~$f$ and satisfies the Energy Dissipation Equality \eqref{eq:ede}.
    In particular it satisfies the Energy Dissipation Inequality \eqref{eq:edi}.
\end{prop}

\begin{proof}
     According to Proposition \ref{prop:fdecreases} and Proposition \ref{prop:prop}, $(y_t)_{t\in (0,T)}$ is locally Lipschitz and the function $t\mapsto f(y_t)$ is locally bounded above. The conclusion then follows from Proposition 3.4 in \cite{ohta}.
\end{proof}

\begin{prop}
    If $f\colon X \to \bar\R$ is a $(K,N)$-convex functional, then for every $y\in \mathcal{D}[f]$ and every $R>0$, with $R<\pi \sqrt{ N/K}$ if $K<0$, it holds that 
    \begin{equation}\label{eq:descendingslope}
        |D^- f|(y) = \sup_{z\in B_R(y)\setminus \{y\}} \paren{ \frac{N}{\skn(\di(z,y))} \braket{1- \frac{f_N(z)}{f_N(y)} }-K\, \frac{\skn(\di(z,y)/2)}{\ckn(\di(z,y)/2)} }^-
    \end{equation}
\end{prop}

\begin{proof}
Fix any $y\in \mathcal{D}[f]$ and $R>0$. We start by observing that, for every~$z$ in a neighborhood of~$y$,
\begin{equation}
    f_N(z)-f_N(y) = e^{- f(z)/N} -e^{- f(y)/N} = -\frac{1}{N} f_N(y) \big( f(z)-f(y)\big) + o\big( f(z)-f(y)\big),
\end{equation}
and, as a consequence, we deduce that
    \begin{equation}\label{eq:dslope}
    |D^- f|(y) =  -\frac{N}{f_N(y)} \limsup_{z\to y} \bigg[\frac{(f_N(z)-f_N(y))^-}{\di(y,z)}\bigg].
\end{equation}
Moreover, recalling that
\begin{equation}
    \skn \big(\di(y,z)\big) = \di(y,z) + o(\di(y,z)) \qquad \text{and} \qquad \frac{\skn \big(\di(y,z)/2\big)}{\ckn \big(\di(y,z)/2\big)} = o(\di(y,z)),
\end{equation}
we conclude that 
\begin{equation}
    |D^- f|(y) = \limsup_{z\to y} \bigg( \frac{N}{\skn(\di(z,y))} \braket{1- \frac{f_N(z)}{f_N(y)} }-K\, \frac{\skn(\di(z,y)/2)}{\ckn(\di(z,y)/2)}\bigg)^-,
\end{equation}
proving that $|D^- f|(y)$ is smaller than or equal to the right-hand side of \eqref{eq:descendingslope}. In order to show the opposite inequality, we fix $z\in B_R(y)\setminus \{y\}$ and consider a geodesic $\gamma\colon [0,1]\to X$ connecting~$y$ and~$z$, such that
\begin{equation}\label{eq:convforz}
    f_N(\gamma_t) \leq \sigma_{K,N}^{(1-t)}(\di) f_N(y) + \sigma_{K,N}^{(t)}(\di) f_N(z) \qquad t\in [0,1],
\end{equation}
where $\di\eqdef \di(y,z)$. We observe that
\begin{equation}
    1-\sigma_{K,N}^{(1-t)}= \frac{\skn(\di)- \skn((1-t)\di)}{\skn(\di)} = \frac{\ckn(\di)t\di+ o(t)}{\skn(\di)},
\end{equation}
thus we have that
\begin{equation}\label{eq:convergenceofcoeff}
    \lim_{t\to 0}\frac{\sigma_{K,N}^{(t)}(\di)}{t \di} = \frac{1}{\skn(\di)} \qquad \text{and} \qquad \lim_{t\to 0}\frac{1-\sigma_{K,N}^{(1-t)}(\di)}{t \di} = \frac{\ckn(\di)}{\skn(\di)} .
\end{equation}
Then, combining \eqref{eq:convforz} and \eqref{eq:convergenceofcoeff}, we deduce that
\begin{align}
    -\frac{N}{f_N(y)} \frac{f_N(\gamma_t) - f_N(y) }{\di(y,\gamma_t) } \leq &-\frac{N}{f_N(y)} \frac{ \sigma_{K,N}^{(t)}(\di) f_N(z) - (1-\sigma_{K,N}^{(1-t)}(\di)) f_N(y) }{t\di}.
\end{align}
Letting~$t\to 0$ in the right-hand side above, we obtain
\begin{align*}
        \lim_{t\downarrow 0}-\frac{N}{f_N(y)} \frac{ \sigma_{K,N}^{(t)}(\di) f_N(z) - (1-\sigma_{K,N}^{(1-t)}(\di)) f_N(y) }{t\di}= &-\frac{N}{f_N(y)} \bigg[ \frac{1}{\skn(\di)} f_N(z) - \frac{\ckn(\di)}{\skn(\di)} f_N(y)\bigg]\\
        = & -\frac{N}{f_N(y)} \bigg[ \frac{f_N(z) - f_N(y)}{\skn(\di)} - \frac{1-\ckn(\di)}{\skn(\di)} f_N(y)\bigg] \\
        =& \frac{N}{\skn(\di)} \braket{1- \frac{f_N(z)}{f_N(y)}}-K\,  \frac{\skn(\di/2)}{\ckn(\di/2)} f_N(y),
\end{align*}
where the last step follows from the half-angle trigonometric formula for the tangent. Finally, keeping in mind \eqref{eq:dslope}, we conclude that 
\begin{align*}
        |D^- f|(y) &\geq \lim_{t\to0} \bigg[ -\frac{N}{f_N(y)} \frac{(f_N(\gamma_t) - f_N(y) )^-}{\di(y,\gamma_t) } \bigg] 
        \\
        &\geq \bigg(\frac{N}{\skn(\di)} \braket{1- \frac{f_N(z)}{f_N(y)} }-K\,  \frac{\skn(\di/2)}{\ckn(\di/2)} f_N(y)\bigg)^-.
\end{align*}
Taking the supremum of the previous inequality over $z\in B_R(y)\setminus \{y\}$, we conclude that $|D^- f|(y)$ is larger than or equal to the right-hand side above.
Together with the converse inequality shown above, this concludes the proof.
\end{proof}

In the following result, we prove that the descending slope $|D^- f|$ of a $(K, N)$-convex function~$f$ is a strong upper gradient for $f$, even when $N < 0$. This property was initially established for the class of $K$-convex functions in \cite[Cor. 2.4.10]{AmbGigSav08}, ensuring the same property also for $(K, N)$-functions for $N > 0$. In particular, this property plays a crucial role in proving the existence of curves satisfying the Energy Dissipation Inequality \eqref{eq:edi}, provided that $f$ is also coercive,
see Remark \ref{rmk:convmm} below.

\begin{corollary}\label{c:SlopeLSC}
    Let $(X, \di)$ be a metric space and $f\colon X \to \bar \R$ be a lower semicontinuous $(K,N)$-convex functional. Then, $|D^- f|$ is a strong upper gradient for $f$ and it is lower semicontinuous. 
\end{corollary}

\begin{proof}
Recalling that the descending slope is a weak upper gradient (cf. \cite[Def. 1.2.2]{AmbGigSav08}), the fact that $|D^- f|$ is a strong upper gradient for $f$ follows from the  absolute continuity of $f \circ \gamma$  for any curve $\gamma \in \mathrm{AC}^1_{\mathrm{loc}}((a, b), X)$ satisfying $|D^- f|(\gamma) |\gamma'| \in L^1(a, b)$, which is the property we are going to prove. Without loss of generality, we can assume that the interval $(a, b)$ is bounded and that the curve $\gamma$ can be extended by continuity to $[a, b]$.
We denote by $\Gamma \coloneqq \gamma ([a, b])$ the compact metric space whose distance is the one induced by $X$, we note that $\Gamma \subset B_R(\gamma(\bar t))$ for some $\bar t \in (a, b)$, and we assume  $R < \pi \sqrt{N/K \,}$ if $K < 0$. Hence we define the global slope  of $f$ related to the curve $\gamma$ by setting
\[
\mathfrak{l}^{\Gamma}_f (y) := \sup_{z \in \Gamma\setminus\{y\}}\frac{\big(f(z) - f(y)\big)^-}{\di(z, y)}.
\]
Making use of \eqref{eq:descendingslope}, we can see that
\[
\mathfrak{l}^\Gamma_f (y) \le |D^- f|(y) + |K| \max\bigg\{\frac{\skn(R/2)}{\ckn(R/2)}, \frac{\mathfrak{s}_{-K, N}(R/2)}{\mathfrak{c}_{-K, N}(R/2)}\bigg\}, \qquad \forall y \in \Gamma.
\]
In particular, we have that $\mathfrak{l}^\Gamma_f (\gamma) |\gamma'| \in L^1(a, b)$ and, since  $\mathfrak{l}^\Gamma_f (y)$ is a strong upper gradient by \cite[Thm. 1.2.5]{AmbGigSav08}, this ensures the absolute continuity of $f \circ \gamma$.

Finally, the representation of $|D^-f|$ in \eqref{eq:descendingslope} also guarantees its lower semicontinuity, following the argument in \cite[Thm 1.2.5]{AmbGigSav08}.
\end{proof}

Let us now show a partial converse to Proposition~\ref{p:EVI-MaxSlope}.

\begin{theorem}[Curves of maximal slope and $\EVI_{K,N}(f)$-gradient curves]\label{thm:maxslopeEVI}
Let~$f\colon X\to \bar \R$ be proper lower semi-continuous functional, and~$\seq{y_t}_{t\in (0,T)}$, for some~$T\in (0,\infty]$, be a curve of maximal slope for~$f$ and starting at~$y_0\in X$. 
Further assume that an~$\evikn$ gradient curve~$\seq{y'_t}_{t\in (0,T)}$ for~$f$ exists, for some $K\in \R$ and $N<0$, starting at~$y_0$, in a set $D\subset \mathcal D [f]$ which contains the image of $\seq{y_t}_{t\in (0,T)}$.

If either~$K\geq 0$ or~$f$ is $(K,N)$-convex and bounded above, then~$\seq{y_t}_{t\in (0,T)} = \seq{y'_t}_{t\in (0,T)}$.
\end{theorem}

\begin{proof}
Since~$\seq{y_t}_{t\in (0,T)}$ is of maximal slope, by~\eqref{eq:maxslope} we have~$\frac{\de}{\de t} f(y_t)\leq 0$ $\Leb^1$-a.e., thus~$t\mapsto f(y_t)$ is non-increasing in light of Definition~\ref{d:EDI}\ref{i:d:EDI:2}.
We deduce that $\seq{y_t}_{t\in (0,T)} \in \mathcal{C}'$ by definition~\eqref{eq:CPrimo} of~$\mathcal{C}'$.

Consider the reparametrizations $\mathcal R_1\tparen{\seq{y_t}_{t\in (0,T)}}$ and $\mathcal R_1\tparen{\seq{y'_t}_{t\in (0,T)}}$.
Since~$\mathcal R_1$ is injective on~$\mathcal C'$ by Proposition~\ref{prop:inverse}, it suffices to show that $\mathcal R_1\tparen{\seq{y_t}_{t\in (0,T)}}=\mathcal R_1\tparen{\seq{y'_t}_{t\in (0,T)}}$.

If~$K\geq 0$, then~$\mathcal R_1\tparen{\seq{y'_t}_{t\in (0,T)}}\in \EVI_0(f_N)\cap \mathcal C''_N$ by Proposition~\ref{prop:correspondence}.
In particular,~$\mathcal R_1\tparen{\seq{y'_t}_{t\in (0,T)}}$ is an $\evi_0$-gradient flow for~$f_N$.
Analogously, if~$f$ is $(K,N)$-convex and bounded above by~$M$, then $\mathcal R_1\tparen{\seq{y'_t}_{t\in (0,T)}}\in \EVI_{\tilde M}(f_N)\cap \mathcal C''_N$ by Proposition~\ref{prop:correspondence2} for~$\tilde M\eqdef -\tfrac{K}{N}e^{-M/N}$.
In particular,~$\mathcal R_1\tparen{\seq{y'_t}_{t\in (0,T)}}$ is an $\evi_{\tilde M}$-gradient flow for~$f_N$.
Thus, in both cases, it suffices to show that~$\seq{y'_t}_{t\in (0,T)}$ satisfies~\eqref{eq:edi} and apply the identification result for~$\evi_\lambda$ gradient curves in~\cite[Thm.~4.2]{MurSav}.

To this end, set~$\seq{z_{t}}_{t\in (0,S)}\eqdef \mathcal R_1\tparen{\seq{y_t}_{t\in (0,T)}}$.
On the one hand, for~$\varphi=\alpha^{-1}$ with~$\alpha$ as in~\eqref{eq:AlphaReparametrization},
\begin{align}\label{eq:p:EVI-EDI:1}
\abs{\dot z_t}=\abs{\dot y_{\varphi(t)}} \varphi'(t) \qquad \text{and} \qquad \abs{D^- f_N}(z_t) = -\frac{f_N(z_t)}{N} \abs{D^- f}(z_t) = \varphi'(t) \abs{D^- f}(z_t)
\end{align}
by~\eqref{eq:dslope} and~\eqref{eq:phi}.
On the other hand,
\begin{align}\label{eq:p:EVI-EDI:2}
-\frac{\de}{\de t} f_N(z_t) &= -\frac{f_N(z_t)}{N} \braket{-\frac{\de}{\de t} f(z_t)}=\varphi'(t) \braket{-\frac{\de}{\de t} f(z_t)}
\end{align}
and, by~\eqref{eq:maxslope},
\begin{align}\label{eq:p:EVI-EDI:3}
-\frac{\de}{\de t} f(z_t) = -\frac{\de}{\de t} f(y_{\varphi(t)}) = \varphi'(t)\frac{\de}{\de s} f(y_s) \big\rvert_{s=\varphi(t)} \geq \varphi'(t) \braket{\frac{1}{2}\abs{\dot y_{\varphi(t)}}^2+\frac{1}{2}\abs{D^- f}^2(y_{\varphi(t)})} .
\end{align}

Combining~\eqref{eq:p:EVI-EDI:1}-\eqref{eq:p:EVI-EDI:3},
\begin{align}\label{eq:p:EVI-EDI:4}
-\frac{\de}{\de t} f_N(z_t) \geq \varphi'(t)^2 \braket{\frac{1}{2}\abs{\dot y_{\varphi(t)}}^2+\frac{1}{2}\abs{D^- f}^2(y_{\varphi(t)})}= \frac{1}{2}\abs{\dot z_t}^2 + \frac{1}{2}\abs{D^-f_N}^2(z_t).
\end{align}
This shows that~$f_N$ satisfies the pointwise Energy Dissipation Inequality~\eqref{eq:maxslope} along~$\seq{z_{t}}_{t\in (0,S)}$.
Furthermore, integrating~\eqref{eq:p:EVI-EDI:4} from an arbitrary~$\epsilon>0$ to~$T$, we see that~$\seq{z_{t}}_{t\in (0,S)}\in \mathrm{AC}^2_\mathrm{loc}((0,T);X)$.
Finally, since $f\tparen{\seq{y_{t}}_{t\in (0,T)}}\in \mathrm{AC}_\mathrm{loc}((0,T);X)$, we claim that~$f_N\tparen{\seq{z_{t}}_{t\in (0,S)}}\in \mathrm{AC}_\mathrm{loc}((0,S);X)$.
Indeed, since~$\seq{y_{t}}_{t\in (0,T)} \in \mathcal{C}'$, (see the beginning of the proof) $f\tparen{\seq{y_{t}}_{t\in (0,T)}}$ is locally bounded in~$(0,T)$ and therefore, for every compact interval~$I \Subset (0,T)$, there exists a constant~$C(I)$ such that 
\begin{equation}
    |f_N(y_t)-f_N(y_s)| \leq C(I)|f(y_t)-f(y_s)|.
\end{equation}
As a consequence, we have that $f_N\tparen{\seq{y_{t}}_{t\in (0,T)}}\in \mathrm{AC}_\mathrm{loc}((0,T);X)$.
Finally, the reparametrization map~$\varphi$ defining~$\seq{z_{t}}_{t\in (0,S)}$ is locally Lipschitz on $(0,S)$ (cf.~\eqref{eq:phi}), thus we can conclude that $f_N\tparen{\seq{z_{t}}_{t\in (0,S)}}\in \mathrm{AC}_\mathrm{loc}((0,S);X)$.

Combining the properties above shows that~$\seq{z_t}_{t\in (0,S)}$ is a curve of maximal slope for~$f_N$ and thus concludes the proof.
\end{proof}

\begin{remark}[Convergence of minimizing-movement schemes]\label{rmk:convmm}
The lower semicontinuity of~$|D^- f|$ established in Corollary~\ref{c:SlopeLSC} allows us to apply the theory of generalized minimizing movement schemes to conclude the \emph{existence} of curves of maximal slope for lower semicontinuous $(K,N)$-convex functionals, $N<0$.

To this end, it suffices to verify the assumptions in~\cite[Thm.s~2.3.1 and 2.3.3, p.~46]{AmbGigSav08}.
Indeed, in the metric setting of~\cite[Rmk.~2.1.1, p.~43]{AmbGigSav08} let~$f\colon X\to \bar\R$ be a proper lower semicontinuous (cf.~\cite[Eqn.~(2.1.2a), p.~43]{AmbGigSav08}),  $(K,N)$-convex, and coercive in the sense of~\cite[Eqn.~(2.1.2b)]{AmbGigSav08}.
We recall that a $(K,N)$-convex functions is not necessarily coercive for~$N<0$ (see Rmk.~\ref{r:NoCoercive}). However, there are examples of $(K,N)$-convex functions bounded below, cf.\ e.g.\ Example~\ref{ex:KNconvex}\ref{i:ex:KNconvex:1}, hence, in particular, coercive.
Since~$|D^- f|$ is lower semicontinuous, it coincides with its lower semicontinuous envelope and so it is a weak upper gradient of~$f$ in the sense of~\cite[Dfn.~1.2.2, p.~27]{AmbGigSav08} by~\cite[Thm.~1.2.5, p.~28]{AmbGigSav08}, and in fact a strong upper gradient, again by Corollary~\ref{c:SlopeLSC}.

Now, it follows from~\cite[Thm.~2.3.3, p.~46]{AmbGigSav08} that the minimizing movement scheme for~$f$ converges to a curve of maximal slope in the sense of~\cite[Dfn.~1.3.2, p.~32]{AmbGigSav08} and therefore to a curve of maximal slope in the sense of Definition~\ref{d:EDI} by~\cite[Rmk.~1.3.3, p.~32]{AmbGigSav08} and the fact that~$|D^- f|$ is a strong upper gradient for~$f$.
\end{remark}

\subsection*{Conflict Of Interest and Data Availability}
The authors certify that they have NO affiliations with or involvement in any
organization or entity with any financial interest or non-financial interest in
the subject matter or materials discussed in this manuscript. Moreover, this manuscript has no associated data.

\bibliographystyle{alpha} 
\bibliography{bibliographyGF}

\begin{thebibliography}{DGMT80}

\bibitem[AGS08]{AmbGigSav08}
Luigi Ambrosio, Nicola Gigli, and Giuseppe Savar\'e.
\newblock {\em Gradient flows in metric spaces and in the space of probability
  measures}.
\newblock Lectures in Mathematics ETH Z\"urich. Birkh\"auser Verlag, Basel,
  second edition, 2008.

\bibitem[AGS14a]{AGSCalc}
Luigi Ambrosio, Nicola Gigli, and Giuseppe Savar\'e.
\newblock Calculus and heat flow in metric measure spaces and applications to
  spaces with {R}icci bounds from below.
\newblock {\em Invent. Math.}, 195(2):289--391, 2014.

\bibitem[AGS14b]{AGS14}
Luigi Ambrosio, Nicola Gigli, and Giuseppe Savar\'e.
\newblock Metric measure spaces with {R}iemannian {R}icci curvature bounded
  from below.
\newblock {\em Duke Math. J.}, 163(7):1405--1490, 2014.

\bibitem[Ama16]{Ama16}
Shun-ichi Amari.
\newblock {\em {Information Geometry and Its Applications}}.
\newblock Springer Japan, 2016.

\bibitem[AMS19]{AMS19}
Luigi Ambrosio, Andrea Mondino, and Giuseppe Savar\'e.
\newblock Nonlinear diffusion equations and curvature conditions in metric
  measure spaces.
\newblock {\em Mem. Amer. Math. Soc.}, 262(1270):v+121, 2019.

\bibitem[BBI01]{BurBurIva07}
Dmitri Burago, {\relax Yu}ri Burago, and Sergei Ivanov.
\newblock {\em {A course in metric geometry}}, volume~33 of {\em {Graduate
  Studies in Mathematics}}.
\newblock American Mathematical Society, Providence, RI, 2001.

\bibitem[BCR05]{BCR05}
F.~Barthe, P.~Cattiaux, and C.~Roberto.
\newblock Concentration for independent random variables with heavy tails.
\newblock {\em AMRX Applied Mathematics Research eXpress}, (2):39--60, 2005.

\bibitem[BL76]{BraLieb}
H.J. Brascamp and E.H Lieb.
\newblock On extensions of the {B}runn-{M}inkowski and {P}r\'ekopa-{L}eindler
  theorems, including inequalities for log concave functions, and with an
  application to the diffusion equation.
\newblock {\em J. Funct. Anal.}, (22):366--389, 1976.

\bibitem[BL09]{BobLed}
S.~G. Bobkov and Michel Ledoux.
\newblock Weighted {P}oincar\'e-type inequalities for {C}auchy and other convex
  measures.
\newblock {\em Ann. Probab.}, (37):403--427, 2009.

\bibitem[Bob07]{Bobkov07}
S.~G. Bobkov.
\newblock Large deviations and isoperimetry over convex probability measures
  with heavy tails.
\newblock {\em Electron. J. Prob.}, (12):1072--1100, 2007.

\bibitem[Bor75]{Borell}
C.~Borell.
\newblock Convex set functions in d-space.
\newblock {\em Period.Math.Hungar}, (6):111--136, 1975.

\bibitem[Bre73]{BrezisBook}
Haim Brezis.
\newblock {\em Op\'erateurs maximaux monotones et semi-groupes de contractions
  dans les espaces de {H}ilbert}, volume No. 5 of {\em North-Holland
  Mathematics Studies}.
\newblock North-Holland Publishing Co., Amsterdam-London; American Elsevier
  Publishing Co., Inc., New York, 1973.
\newblock Notas de Matem\'atica, No. 50. [Mathematical Notes].

\bibitem[DG93]{DeG}
Ennio De~Giorgi.
\newblock New problems on minimizing movements.
\newblock In {\em Boundary value problems for partial differential equations
  and applications}, volume~29 of {\em RMA Res. Notes Appl. Math.}, pages
  81--98. Masson, Paris, 1993.

\bibitem[DGMT80]{DeGMaTo}
Ennio De~Giorgi, Antonio Marino, and Mario Tosques.
\newblock Problems of evolution in metric spaces and maximal decreasing curve.
\newblock {\em Atti Accad. Naz. Lincei Rend. Cl. Sci. Fis. Mat. Nat. (8)},
  68(3):180--187, 1980.

\bibitem[DMT85]{DeMaTo}
Marco Degiovanni, Antonio Marino, and Mario Tosques.
\newblock Evolution equations with lack of convexity.
\newblock {\em Nonlinear Anal.}, 9(12):1401--1443, 1985.

\bibitem[DS08]{DanSav}
Sara Daneri and Giuseppe Savar\'{e}.
\newblock Eulerian calculus for the displacement convexity in the {W}asserstein
  distance.
\newblock {\em SIAM J. Math. Anal.}, 40(3):1104--1122, 2008.

\bibitem[DS14]{DaSa}
Sara Daneri and Giuseppe Savar\'e.
\newblock Lecture notes on gradient flows and optimal transport.
\newblock In {\em Optimal transportation}, volume 413 of {\em London Math. Soc.
  Lecture Note Ser.}, pages 100--144. Cambridge Univ. Press, Cambridge, 2014.

\bibitem[DSMP24]{LzDSMaaPed23}
Lorenzo Dello~Schiavo, Jan Maas, and Francesco Pedrotti.
\newblock {Local Conditions for Global Convergence of Gradient Flows and
  Proximal Point Sequences in Metric Spaces}.
\newblock {\em {Trans.\ Amer.\ Math.\ Soc.}}, 377(6):3779--3804, 2024.

\bibitem[EKS15]{ErbarKuwadaSturm}
Matthias Erbar, Kazumasa Kuwada, and Karl-Theodor Sturm.
\newblock On the equivalence of the entropic curvature-dimension condition and
  {B}ochner's inequality on metric measure spaces.
\newblock {\em Invent. Math.}, 201(3):993--1071, 2015.

\bibitem[GKO13]{GKO13}
Nicola Gigli, Kazumasa Kuwada, and Shin-Ichi Ohta.
\newblock Heat flow on {A}lexandrov spaces.
\newblock {\em Comm. Pure Appl. Math.}, 66(3):307--331, 2013.

\bibitem[Kol14]{Kol13}
Alexander~V. Kolesnikov.
\newblock Hessian metrics, {CD(K,N)}-spaces, and optimal transportation of
  log-concave measures.
\newblock {\em American Institute of Mathematical Sciences}, 34(4):1511--1532,
  2014.

\bibitem[McC94]{McCann}
Robert~John McCann.
\newblock {\em A convexity theory for interacting gases and equilibrium
  crystals}.
\newblock ProQuest LLC, Ann Arbor, MI, 1994.
\newblock Thesis (Ph.D.)--Princeton University.

\bibitem[Mil17]{Milman-NegSph}
Emanuel Milman.
\newblock Harmonic measures on the sphere via curvature-dimension.
\newblock {\em Ann. Fac. des Sc. de Toulouse}, 26(2):437--449, 2017.

\bibitem[MR23]{MR23}
Mattia Magnabosco and Chiara Rigoni.
\newblock Optimal maps and local-to-global property in negative dimensional
  spaces with {R}icci curvature bounded from below.
\newblock {\em Tohoku Math. J. (2)}, 75(4):483--507, 2023.

\bibitem[MRS13]{MiRoSa}
Alexander Mielke, Riccarda Rossi, and Giuseppe Savar\'e.
\newblock Nonsmooth analysis of doubly nonlinear evolution equations.
\newblock {\em Calc. Var. Partial Differential Equations}, 46(1-2):253--310,
  2013.

\bibitem[MRS23]{MRS23}
Mattia Magnabosco, Chiara Rigoni, and Gerardo Sosa.
\newblock Convergence of metric measure spaces satisfying the cd condition for
  negative values of the dimension parameter.
\newblock {\em Nonlinear Analysis}, 237:113366, 2023.

\bibitem[MS20]{MurSav}
Matteo Muratori and Giuseppe Savar\'{e}.
\newblock Gradient flows and evolution variational inequalities in metric
  spaces. {I}: {S}tructural properties.
\newblock {\em J. Funct. Anal.}, 278(4):108347, 67, 2020.

\bibitem[Oht16]{ohta}
Shin-ichi Ohta.
\newblock {$(K,N)$}-convexity and the curvature-dimension condition for
  negative~{$N$}.
\newblock {\em J. Geom. Anal.}, 26(3):2067--2096, 2016.

\bibitem[OT11]{OTI}
Shin-ichi Ohta and Asuka Takatsu.
\newblock Displacement convexity of generalized relative entropies.
\newblock {\em Adv. Math.}, 228(3):1742--1787, 2011.

\bibitem[OT13]{OTII}
Shin-Ichi Ohta and Asuka Takatsu.
\newblock Displacement convexity of generalized relative entropies. {II}.
\newblock {\em Comm. Anal. Geom.}, 21(4):687--785, 2013.

\bibitem[RS06]{RoSa}
Riccarda Rossi and Giuseppe Savar\'e.
\newblock Gradient flows of non convex functionals in {H}ilbert spaces and
  applications.
\newblock {\em ESAIM Control Optim. Calc. Var.}, 12(3):564--614, 2006.

\bibitem[San17]{Sant}
Filippo Santambrogio.
\newblock {Euclidean, metric, and Wasserstein gradient flows: an overview}.
\newblock {\em Bull. Math. Sci.}, 7(1):87--154, 2017.

\bibitem[Stu06]{Sturm06II}
Karl-Theodor Sturm.
\newblock On the geometry of metric measure spaces. {II}.
\newblock {\em Acta Math.}, 196(1):133--177, 2006.

\bibitem[Stu18]{Sturm18}
Karl-Theodor Sturm.
\newblock Gradient flows for semiconvex functions on metric measure
  spaces---existence, uniqueness, and {L}ipschitz continuity.
\newblock {\em Proc. Amer. Math. Soc.}, 146(9):3985--3994, 2018.

\bibitem[Vil09]{Villani}
C\'edric Villani.
\newblock {\em Optimal transport}, volume 338 of {\em Grundlehren der
  mathematischen Wissenschaften [Fundamental Principles of Mathematical
  Sciences]}.
\newblock Springer-Verlag, Berlin, 2009.
\newblock Old and new.

\end{thebibliography}

\end{document}